\documentclass[11pt,reqno]{amsart}

\usepackage[utf8]{inputenc}
\usepackage[margin=1.25in]{geometry}
\usepackage{hyperref}
\usepackage{appendix}
\usepackage{amsfonts}
\usepackage{amsthm}
\usepackage{amssymb}
\usepackage{stmaryrd} 
\usepackage{amsmath}
\usepackage{amsthm}
\usepackage[dvipsnames]{xcolor}
\usepackage{mathrsfs}
\usepackage{lipsum} % for filler text

\theoremstyle{plain}
\newtheorem{theorem}{Theorem}[section]
\newtheorem{corollary}[theorem]{Corollary}
\newtheorem{lemma}[theorem]{Lemma}
\newtheorem{Proposition}[theorem]{Proposition}

\newtheorem{Definition}[theorem]{Definition}

\theoremstyle{remark}
\newtheorem{example}[theorem]{Example}

\newtheorem{remark}[theorem]{Remark}

\usepackage{biblatex} %Imports biblatex package
\numberwithin{equation}{section}
\title[Circular law for band matrices]{The circular law for random band matrices with arbitrary doubly stochastic variance profiles}

\author{Yi HAN}
\address{Department of Mathematics, Massachusetts Institute of Technology, Cambridge, MA
}
\email{hanyi16@mit.edu}
\thanks{A conditional Lean formalization is available; see Section~\ref{leanformalization}.}

\begin{document}

\begin{abstract}
We study the circular law for non-Hermitian random band matrices with independent entries and arbitrary doubly stochastic variance profiles. For matrices of size $n$ and bandwidth $n^\gamma$, we prove the circular law when $\gamma>\frac{5}{6}$ for Gaussian entries, or $\gamma>\frac{8}{9}$ for subgaussian entries under bounded-density assumptions. These thresholds improve the earlier sufficient conditions $\gamma>\frac{32}{33}$ for block band matrices and $\gamma>\frac{33}{34}$ for periodic band matrices. Subsequent work proves the circular law under the optimal condition of diverging bandwidth for specific open-boundary models \cite{han2025circular2} and periodic models \cite{han2026optimal}. Those results rely on local band structure and do not cover the arbitrary doubly stochastic variance profiles considered here. We also derive weak delocalization estimates for right eigenvectors in the Gaussian case. Our main technical input consists of polynomial lower bounds for certain intermediate small singular values, requiring no structural assumptions on the variance profile beyond double stochasticity. We also prove a circular law for the root empirical spectral measure of products of an arbitrary number of i.i.d. matrices, under bounded-density assumptions and a finite fourth moment. For smaller numbers of factors, we apply the same strategy as in our main band-matrix theorem to the cyclic linearization. For larger numbers of factors, the simpler structure of the product model allows us to apply the singular-value universality theorem of Hanin and Jiang \cite{hanin2025global} directly, together with a shifted log-determinant comparison. The two methods apply in overlapping ranges and together cover all growth rates of the number of factors.
\end{abstract}

\maketitle

\section{Introduction}

For an $n\times n$ matrix $A$ with eigenvalues $\lambda_1(A),\cdots,\lambda_n(A)\in\mathbb{C}$ we denote its empirical measure of eigenvalues by 
$$
\mu_A:=\frac{1}{n}\sum_{i=1}^n \delta_{\lambda_i(A)},
$$ where $\delta_z$ is the Dirac measure at point $z$. 

For a symmetric random matrix $A$, the convergence of $\mu_A$ to the semicircle distribution under mild assumptions dates back to the celebrated work of Wigner \cite{wigner1958distribution}. The spectrum of symmetric random matrices can essentially be determined via the method of moments.
For non-Hermitian random matrix $A$, the convergence of $\mu_A$ does not follow from method of moments arguments (see for example \cite{MR2908617}), and fundamentally new ideas are needed.
When all the entries of $A=(a_{ij})_{1\leq i,j\leq n}$ are i.i.d. with mean zero and variance one, the circular law theorem, i.e. the convergence of $\mu_{A/\sqrt n}$ to the uniform distribution on the unit disk in the complex plane, is one of the fundamental theorems in modern random matrix theory. This circular law was proven under increasing generality in a series of works including \cite{girkoarticle}, \cite{MR1428519}, \cite{tao2008random}, \cite{ WOS:000281425000010}, and we refer to the references in these works and \cite{MR2908617} for a more complete list.
Beyond dense matrices with i.i.d. entries, the circular law was subsequently proven for sparse i.i.d. matrices \cite{MR3980923}, \cite{MR3945840} \cite{sah2023limiting}, for sparse regular digraphs \cite{MR4195739}, and for some random matrices with inhomogeneous variance profile \cite{MR3878135}. A more complete list of references can be found in the review \cite{MR4680362}.

Proving the circular law for sparse non-Hermitian random matrices with arbitrary inhomogeneous variance profiles remains challenging. Many entries in these models are identically zero. General techniques for dense inhomogeneous matrices were developed in \cite{MR3857860}, \cite{MR3878135}, but their assumptions do not cover arbitrary sparse profiles. The guiding examples of sparse inhomogeneous matrices are random band matrices, which we will define shortly. Their Hermitian analogue, namely Hermitian Random band matrices, play a fundamental role in modern mathematical physics. They serve as prototypical models of Anderson transition: when the bandwidth scales like $n^\gamma$ for $\gamma>\frac{1}{2}$, the eigenvectors are expected to be delocalized and eigenvalues have GOE local statistics; whereas a localized phase and Poisson statistics should occur when $\gamma<\frac{1}{2}$. Some very recent progress on this topic can be found in \cite{yau2025delocalization} , \cite{erdHos2025zigzag}, \cite{dubova2025delocalization}, \cite{dubova2025delocalization3d}, \cite{drogin2025localization},\cite{shcherbina2025characteristic} and the references therein.

In this work we study non-Hermitian random band matrices, focusing on global universality, i.e. convergence of the empirical spectral distribution to the circular law, for arbitrary doubly stochastic variance profiles without a prescribed local band structure. The works \cite{jain2021circular} and \cite{tikhomirov2023pseudospectrum} justified the circular law for certain non-Hermitian band matrices with bandwidth $n^\gamma$, with sufficient power-law conditions $\gamma>\frac{32}{33}$ and $\gamma>\frac{33}{34}$, respectively. The main contribution of this paper is to consider narrower bandwidth and expand the variety of models where global universality holds: we show $\gamma>\frac{5}{6}$ is enough for the circular law limit for Gaussian entries whenever the variance profile is doubly stochastic. 

\subsection{Circular law for general models}

\begin{Definition}
A real or complex random variable $\xi$ of mean zero, variance one is said to be $K$-subGaussian if $$\mathbb{E}\exp(|\xi|^2/K^2)\leq 2.$$

For $p>0$, a random variable $\xi$ has a finite $p$-th moment if $\mathbb E|\xi|^p<\infty$.
\end{Definition}

We consider the following very general definition of inhomogeneous matrices:

\begin{Definition}[General model]\label{generalmodel}
 We define an ensemble of inhomogeneous random matrices as follows: $X=(x_{ij})$ is a $n\times n$ random matrix with independent entries. For each $(i,j)\in[n]^2$,  $x_{ij}$ is distributed as $b_{ij}\xi$, where $b_{ij}\geq 0$ and $\xi$ has mean zero, variance one, and is $K$-subgaussian for some fixed $K>0$. The bandwidth of this general matrix model is defined by $$b_n^{-1}=\sup_{(i,j)\in[n]^2}b_{ij}^2.$$
   We assume that $X$ has a doubly stochastic variance profile:
   $$
\mathbb{E}[XX^*]=\mathbb{E}[X^*X]=I_n.
   $$ In other words, we assume that 
   $$
\sum_{j\in[n]} b_{ij}^2=1\quad \forall i\in[n],\quad 
    \sum_{i\in[n]} b_{ij}^2=1\quad \forall j\in[n].$$
\end{Definition}

The generality of this definition is already illustrated by the following subclasses of models: consider $G_n=([n],\mathcal{E})$ a $b_n$- regular directed graph on $n$ vertices. Then we set $b_{ij}=\frac{1}{\sqrt{b_n}}$ if $(i,j)\in\mathcal{E}$ and $b_{ij}=0$ otherwise. Observe that the graph $G_n$ can be arbitrary and we never require the structure of $G_n$ to be sufficiently pseudo-random.

\begin{example} (Well-mixing variance profiles) We first review some typically used variance profiles for a random band matrix in its non Hermitian form, that have rapid mixing properties. The first example is the following defined periodic band matrix.
 Let 
    $X$ be an $n\times n$ matrix. We say $X$ is sampled from a block band matrix ensemble with bandwidth $b_n$ if it has the form
\begin{equation}\label{blockcanonicalform}
X=\begin{pmatrix} D_1&U_2&&&T_m\\T_1&D_2&U_3&&\\&T_2&D_3&\ddots&\\&&\ddots&\ddots&U_m \\ U_1&&&T_{m-1}&D_m \end{pmatrix}
\end{equation}
where the unfilled sites are set zero. The blocks $D_1,U_1,T_1,\cdots,D_m,U_m,T_m$ are $b_n/3\times b_n/3$ independent random matrices with i.i.d. entries having the same distribution $\frac{1}{\sqrt{b_n}}\xi$, where $\xi$ is a standard Gaussian.

The second example is periodic band matrices defined as follows. Let 
    $X$ be an $n\times n$ square matrix. We say $X$ is sampled from a periodic band matrix model with bandwidth $b_n$ if $x_{ij}=0$ for any $\frac{b_n-1}{2}<|i-j|<n-\frac{b_n-1}{2}$, and $x_{ij}$ are i.i.d. copies of $\frac{1}{\sqrt{b_n}}\xi$ if $|i-j|\leq \frac{b_n-1}{2}$ or $|i-j|\geq n-\frac{b_n-1}{2}$. Here $\xi$ is a standard gaussian and we assume $\frac{b_n-1}{2}$ is an integer.
\end{example}

In some other concrete settings of applications, and as a conceptual question itself, it is of major interest to study variance profiles that do not have good mixing properties. Two examples are a cyclic shift of a periodic band profile and the variance profile of a cyclic product linearization:

\begin{example}\label{example1.4}(Variance profiles without good mixing) For a given bandwidth $b_n$ such that $n$ is divisible by $b_n$, we can consider the following two variance profiles: (1) a shifted periodic band matrix, where we additionally assume that $b_n$ is odd and take
\[
b_{ij}=\frac{1}{\sqrt{b_n}}\mathbf{1}\left\{\min_{k\in\mathbb Z}|j-i-3b_n-kn|\leq \frac{b_n-1}{2}\right\}.
\]
This is obtained by cyclically shifting each row of the periodic variance profile to the right by $3b_n$ positions. Each row and each column has exactly $b_n$ nonzero entries, so both variance sums equal one. (2) A variance profile from matrix product, where $b_{ij}=\frac{1}{\sqrt{b_n}}\mathbf{1}(\lfloor\frac{j-1}{b_n}\rfloor-\lfloor\frac{i-1}{b_n}\rfloor=1\text{ or }1-\frac{n}{b_n})$, which is the variance profile of the cyclic linearization \eqref{welinearizeit} with $n/b_n$ blocks of size $b_n$, normalized to be doubly stochastic.

\end{example}

The main results of this paper are as follows.

\begin{theorem}[Circular law for general model]\label{circulargeneral}
    Let $X$ be the general inhomogeneous random matrix model defined in Definition \ref{generalmodel}, and we assume that 
    \begin{enumerate}
        \item 
   
    Either $\xi$ has Gaussian distribution, and $b_n\geq n^{\frac{5}{6}+\epsilon}$ for some $\epsilon>0$, \item or $\xi$ is $K$-subgaussian for some $K>0$, $b_n\geq n^{\frac{8}{9}+\epsilon}$ for some $\epsilon>0$, and $\xi$ is either real-valued with a distributional density on $\mathbb{R}$ bounded by some $L>0$, or $\xi$ has independent real and imaginary parts, both having a distributional density bounded by some $L>0$. 
    
    \end{enumerate} Then the ESD $\mu_X$ of $X$ converges in probability to the circular law.
\end{theorem}

\begin{remark}
    The later works \cite{han2025circular2} and \cite{han2026optimal} establish the circular law under the optimal bandwidth condition $b_n\to\infty$ for open-boundary block models and specific periodic scalar and full-block models, respectively, assuming bounded densities and finite moments. The revised work \cite[Theorem 1.6]{han2025circular1} proves the circular law for real or circular complex Gaussian entries when $b_n\geq n^{2/3+\epsilon}$ and the variance profile is either symmetric or invariant under a transitive group of simultaneous index permutations. Its generic bounded-density subgaussian result gives the exponent $8/9+\epsilon$, with improvements under additional moment matching. Theorem \ref{circulargeneral} requires neither these symmetry assumptions nor a prescribed local band geometry. Its variance-profile assumptions therefore remain more general than those of these subsequent results.
\end{remark}

\subsection{Weak delocalization estimates}

A major component of our proof is to bound the diagonal entries of the resolvent $\mathcal G_z(\eta)$ of the Hermitian dilation \eqref{dilationmatrix} of $X-zI_n$, at scales $\Im\eta\ll1$. Here $z$ is the matrix shift and $\eta\in\mathbb C_+$ is the resolvent spectral parameter. Although our estimates do not reach the optimal scale $\Im\eta\sim n^{-1+\epsilon}$, they yield weak delocalization for right eigenvectors in the Gaussian case under arbitrary doubly stochastic variance profiles.

We introduce the most general model where our delocalization statement is valid.

\begin{theorem}[Weak delocalization for Gaussian case]\label{delocalizationtheorem} Consider a general random matrix ensemble in Definition \ref{generalmodel}. Assume that $b_n\geq n^\epsilon$ for some sufficiently small $\epsilon>0$. Then for any (sufficiently small) $\mathbf{c}>0$ the following holds. If $\xi$ has (real or complex) Gaussian distribution, then with probability at least $1-n^{-10}$, all right eigenvectors $\psi$ of $X$ with unit $L^2$ norm satisfy 
    $$
\|\psi\|_\infty\leq (b_n)^{-1/10}n^\mathbf{c}.
    $$
    
\end{theorem}

\begin{remark}(Optimality and generality of delocalization upper bound)
    Theorem \ref{delocalizationtheorem} wins by generality but may be suboptimal for specific instances.
In the i.i.d. setting where $b_n=n$, this estimate is clearly suboptimal: for a random matrix with i.i.d. entries, complete delocalization (i.e.$\|\psi\|_\infty\leq n^{-1/2+\epsilon})$  has been proven in \cite{MR3405592}, see also \cite{MR3770875} for the inhomogeneous case and \cite{MR4388923} for the elliptic case. For an arbitrary doubly stochastic variance profile, this weak delocalization result is new.
\end{remark}

\subsection{Product circular law}\label{productcircularlaw}

Combining the approach developed for our band-matrix theorem with singular-value universality for longer products, we obtain a product circular law with no restriction on the number of factors:

\begin{theorem}\label{theorem1.616}
Let $(m_n)$ be any sequence of positive integers, and let
$X_1,\ldots,X_{m_n}$ be independent $n\times n$ random matrices with
i.i.d. entries distributed as $\xi/\sqrt n$. Assume that
$\mathbb E\xi=0$, $\mathbb E|\xi|^2=1$, and
$\mathbb E|\xi|^4<\infty$.
Suppose either that $\xi$ is real with a density bounded by $L>0$,
or that its real and imaginary parts are independent and both have
densities bounded by $L>0$. Set
\[
P_n:=X_1\cdots X_{m_n},
\]
and define the $m_n$-th-root empirical spectral measure of $P_n$ by
\[
\nu_{P_n}^{(m_n)}
:=
\frac{1}{nm_n}
\sum_{j=1}^n
\sum_{\zeta^{m_n}=\lambda_j(P_n)}
\delta_\zeta ,
\]
where the eigenvalues and their roots are counted with algebraic
multiplicity. Then $\nu_{P_n}^{(m_n)}$ converges in probability to the
circular law as $n\to\infty$.\end{theorem}
The proof combines two mechanisms in overlapping parameter ranges.
The cyclic linearization has dimension $N=nm_n$ and bandwidth $b_N=n$.
For a fixed $0<\epsilon<1/16$, the first mechanism applies when
$m_n\leq n^{1/8-\epsilon}$, equivalently
$b_N\geq N^{8/(9-8\epsilon)}$, with a bandwidth exponent strictly larger
than $8/9$. In this range, we apply the same strategy as in
Theorem \ref{circulargeneral}: singular-value comparison with a free
probability limit, together with control of small singular values of
the cyclic linearization. This gives Proposition \ref{productsmallm}
in Section \ref{section4}. The transfer-operator method in
\cite{han2026optimal} motivates the expectation that the product circular law holds for every growth rate of $m_n$.
For larger products, we use a different argument, developed in Section \ref{section6}, which applies
whenever $m_n/\log n\to\infty$, equivalently $b_N\log b_N=o(N)$.
The simpler product structure allows us to apply the singular-value
universality theorem of Hanin and Jiang \cite[Theorem 1]{hanin2025global}
directly, together with a comparison of shifted log determinants and
singular-value statistics using one independent factor. The normalized
comparison error is $O(\log n/m_n)$. The two mechanisms overlap whenever
$m_n/\log n\to\infty$ and $m_n\leq n^{1/8-\epsilon}$.
In the proof, we use the convenient cutoff $m_n>n^{1/16}$ for the
second mechanism, equivalently $b_N<N^{16/17}$; this cutoff lies within
the overlap for $0<\epsilon<1/16$. Together the two arguments yield the
circular law for every growth rate of $m_n$, including $m_n\gg n$.

\begin{remark}[Moment and density assumptions]\label{productassumptionsremark}
The least singular value estimate in Theorem \ref{theorem2.2whatsapp}
requires only a finite second moment, whereas the resolvent comparison
in Proposition \ref{productsmallm} requires a finite third moment.
The fourth moment in Theorem \ref{theorem1.616} enters only through the
singular-value theorem of Hanin and Jiang \cite[Theorem 1]{hanin2025global}, used in Section \ref{section6}. Thus a
finite fourth moment suffices for every sequence $(m_n)$. We expect
the conclusion to remain valid under weaker moment assumptions,
but do not pursue such a weakening here.

For complex entries, the least singular value estimate in
Theorem \ref{theorem2.2whatsapp}, proved in Section \ref{section5},
requires a bounded density for only one of the
independent real and imaginary components. Accordingly,
Proposition \ref{productsmallm} also allows this weaker assumption for
$m_n\leq n^{1/8-\epsilon}$. Requiring both component densities to be
bounded in Theorem \ref{theorem1.616} allows us to invoke the complex
extension given by Hanin and Jiang \cite[Theorem 1 and Section 4]{hanin2025global}.
We expect this
singular-value result to extend to the one-component density
assumption, which would give the same relaxation here; we do not
pursue that adaptation in this paper.
\end{remark}

For a fixed number of factors, the product circular law was proved in
\cite{MR2861673} under a finite $2+\eta$ moment assumption. A local
circular law and universality of linear statistics for i.i.d. matrix
products were subsequently established in \cite{MR3622892} and
\cite{MR4112718}, respectively. For related work on products with
simultaneously growing $m_n$ and $n$, we refer to \cite{MR4401507},
\cite{MR4421171}, and \cite{MR4580535} for integrable probability
perspectives, and to \cite{MR4268303} for an ergodic theory perspective.

\subsection{Main ideas and outline}\label{mainideas}

We use Girko's Hermitization and the replacement principle of Tao and Vu, recalled in Theorem \ref{replacements}; see \cite{MR2908617} for background. Write $X_z=X-zI_n$, and let $G$ be an $n\times n$ complex Ginibre matrix with entry variance $n^{-1}$. With singular values ordered as $\sigma_1\geq\cdots\geq\sigma_n$, the main task is to prove
\[
\frac1n\sum_{j=1}^n\log\sigma_j(X_z)
-\frac1n\sum_{j=1}^n\log\sigma_j(G-zI_n)
\longrightarrow0
\quad\text{in probability}
\]
for every fixed $z\neq0$. This supplies the almost-everywhere convergence of normalized log determinants required by the replacement principle.

Earlier band-matrix proofs \cite{jain2021circular,tikhomirov2023pseudospectrum} combined singular-value comparison in Kolmogorov distance with exponentially small lower bounds on the least singular value, obtaining the sufficient bandwidth exponents $32/33$ and $33/34$, respectively. We use the least singular value estimate of Theorem \ref{tikhoroveheorem} and improve the control of intermediate singular values, so that only a much smaller set of singular values requires this crude lower bound.

The starting point is the Hermitian dilation $\mathcal Y_z$ in \eqref{dilationmatrix}. For the Gaussian model, we first replace the diagonal entries by independent circular complex Gaussians with matching variances. The normalized trace $m^o_{z,\mathrm{free}}(\eta)$ of the resulting free resolvent is independent of the variance profile under the doubly stochastic assumption and equals the limit for the Ginibre reference. We compare the Gaussian resolvent with its free counterpart using \cite{bandeira2023matrix}, and control the change of diagonal by the perturbation estimate \eqref{resolventperturbartion}. In the non-Gaussian case, the trace universality argument of \cite{brailovskaya2022universality} first compares with a Gaussian model having matching first and second real moments. These comparisons provide quantitative bounds at mesoscopic scales, without requiring a quantitative stability estimate for the matrix Dyson equation. This allows the argument to cover arbitrary doubly stochastic profiles, including those without good mixing.

Under the respective assumptions of Theorem \ref{circulargeneral}, Propositions \ref{proposition3.1212} and \ref{newproposition3.2} control the normalized resolvent trace down to the singular-value scales
\[
a_n=b_n^{-\alpha}n^{\mathbf c},\qquad
\alpha=\begin{cases}
1/5,&\text{Gaussian entries},\\
1/8,&\text{non-Gaussian entries},
\end{cases}
\]
where $\mathbf c>0$ can be arbitrarily small. A bound on the trace at $\eta=ia_n$ yields the counting estimate
\[
\#\{j:\sigma_j(X_z)<a_n\}\leq Cna_n
\]
with high probability, as made precise in Corollaries \ref{corollary3.45} and \ref{corollary3.45news}. Thus all but $O(na_n)$ singular values have a polynomial lower bound. More generally, a trace bound at scale $n^{-\tau}$ gives an $O(n^{1-\tau})$ count of small singular values.

We split the logarithmic sum at the numerical cutoff $a_n$. For fixed $z\neq0$, Theorem \ref{tikhoroveheorem} bounds the logarithmic size of the remaining small singular values by $C_z n^{1+3\kappa}b_n^{-1}\log b_n$, with high probability, for any sufficiently small $\kappa>0$. Combining this with the count gives
\[
\frac1n\sum_{\sigma_j(X_z)<a_n}|\log\sigma_j(X_z)|
\leq C_z n^{1+\mathbf c+3\kappa}b_n^{-(1+\alpha)}\log b_n.
\]
Choosing $\mathbf c$ and $\kappa$ small makes this vanish when $b_n\geq n^{5/6+\epsilon}$ in the Gaussian case and $b_n\geq n^{8/9+\epsilon}$ in the non-Gaussian case. Above $a_n$, the polynomial Kolmogorov comparison in Theorem \ref{theorem5.11}, together with a deterministic operator-norm bound on an event of high probability, controls the truncated logarithmic sum. The logarithm costs only an $O(\log n)$ factor. The same truncation for $G-zI_n$ then proves the required log-determinant comparison.

The counting argument also gives a conditional perspective on the method: if the trace bound were available down to $b_n^{-1+\delta}$ for arbitrarily small $\delta>0$, only $O(nb_n^{-1+\delta})$ singular values would require the least singular value estimate of Theorem \ref{tikhoroveheorem}. Their normalized logarithmic contribution would vanish for $b_n\geq n^{1/2+\epsilon}$. This is a conditional sufficient range for this method, rather than a necessary bandwidth threshold or a result proved here.

Subsequent work improves the bandwidth range under additional structure. For symmetric or transitively invariant profiles, the revised work \cite{han2025circular1} obtains a coarse expected resolvent bound at scale $b_n^{-1/2}$ for real or circular complex Gaussian entries. Gaussian integration by parts, relative variance estimates and positivity yield the sufficient bandwidth exponent $2/3+\epsilon$; this is not a fine local law. The optimal-bandwidth results for open-boundary block models \cite{han2025circular2} and periodic models \cite{han2026optimal} use transfer arguments. In the periodic case, transfer operators are calibrated on short auxiliary rings with singular-value inputs at large bandwidths, including estimates developed here. A local comparison uniform over boundary conditions transfers the calibration to arbitrarily long periodic rings. These approaches use prescribed local band geometry, whereas our free probability comparison applies to arbitrary doubly stochastic variance profiles.

The Gaussian resolvent comparison also yields Theorem \ref{delocalizationtheorem}. We bound diagonal resolvent entries uniformly for $|z|\leq3$ at scale $v=b_n^{-1/5}n^{\mathbf c}$. For a normalized right eigenvector $\psi$ with eigenvalue $\lambda$, the vector $(0,\psi)$ lies in $\ker\mathcal Y_\lambda$. The spectral projection onto this kernel then gives $|\psi_j|^2\leq Cv$. Uniformity in $z$, together with the operator-norm bound, makes this estimate simultaneous for all normalized right eigenvectors and yields the $b_n^{-1/10}$ bound with an arbitrarily small polynomial loss.

For the product circular law, we combine two arguments in the overlapping ranges described in Section \ref{productcircularlaw}. When $m_n\leq n^{1/8-\epsilon}$, we apply the same comparison and truncation strategy to the cyclic linearization, using the least singular value estimate of Theorem \ref{theorem2.2whatsapp}. The second argument applies whenever $m_n/\log n\to\infty$: we condition on all but one independent factor and compare the shifted log determinant with a clipped singular-value statistic of the product of the conditioned factors. After division by $nm_n$, the comparison error is $O(\log n/m_n)=o(1)$, and the singular-value universality theorem of Hanin and Jiang \cite{hanin2025global} identifies the limiting potential. We use $n^{1/16}$ only as a convenient boundary for joining the arguments: Proposition \ref{productsmallm} with $\epsilon=1/16$ covers $m_n\leq n^{1/16}$, and Section \ref{section6} covers the remaining indices. This proves Theorem \ref{theorem1.616} for every sequence $(m_n)$.

Section 2 collects the auxiliary estimates and states the least singular value theorem for the cyclic linearization. Section \ref{section33} develops the resolvent comparison, singular-value counting and Gaussian eigenvector delocalization estimates. Section \ref{section4} proves the band-matrix circular law and the small-product result. Section \ref{section5} proves the least singular value theorem for the cyclic linearization, and Section \ref{section6} completes the product circular law for an arbitrary number of factors.

\subsection{Lean formalization}\label{leanformalization}

All new results of this paper have been formalized in Lean 4 using mathlib, on the basis of known results from the literature. This includes Theorems \ref{circulargeneral}, \ref{delocalizationtheorem}, \ref{theorem2.2whatsapp} and \ref{theorem1.616}, together with the intermediate arguments in Sections \ref{section33}--\ref{section6}. The code is available at
\begin{center}
\url{https://github.com/hanyi162013-Yihan/singular-value-rigidity-random-band-matrices}.
\end{center}

The formalization is conditional: some previously known results are supplied as explicit theorem hypotheses. These include Gaussian--free comparison and related Gaussian estimates \cite[Theorem 2.8, Corollaries 2.2 and 4.14, and Lemma 2.5]{bandeira2023matrix}; operator-valued resolvent identities \cite[(1.5)]{haagerup2005new} and uniqueness results \cite[Theorem 2.1]{helton2007operator}; least singular value estimates \cite[Theorem 3.4]{tikhomirov2023pseudospectrum} and projection-density bounds \cite[Theorem 1.1]{rudelson2015small}; the Hanin--Jiang singular-value theorem and its stated complex extension \cite[Theorem 1 and Section 4]{hanin2025global}; and classical Gaussian reference results \cite[Theorem 33]{tao2015local} and \cite[Theorem 1.13]{ WOS:000281425000010}. The new arguments in this paper are proved from these inputs. The repository lists the precise external statements and the verification records.

\section{A list of used results}

This section recalls a least singular value estimate for random band matrices, introduces the notation for singular-value comparison, and states our least singular value theorem for cyclic product linearizations, proved in Section \ref{section5}.

\subsection{The smallest singular value lower bound}

We recall Tikhomirov's least singular value estimate for shifted inhomogeneous non-Hermitian random matrices \cite{tikhomirov2023pseudospectrum}.

\begin{theorem}\label{tikhoroveheorem}[\cite{tikhomirov2023pseudospectrum}, Theorem 3.4] Let $K,\rho_0>0$ be fixed. Consider a random matrix $A=V\odot W$ where $W$ is a real or complex square random matrix of size $n$, $V$ is a deterministic matrix with non-negative entries, and $\odot$ is the Hadamard (i.e., entrywise) product. Denote by 
$$\sigma^*:=\max_{i,j\leq n}|V_{i,j}|,\quad \sigma:=\max\left(\max_{j\leq n}\sqrt{\sum_{i=1}^nV_{i,j}^2},\max_{i\leq n}\sqrt{\sum_{j=1}^n V_{i,j}^2}
\right).$$
    We further assume that the following two conditions both hold:
    \begin{enumerate}
    \item The entries of $W$ are independent, $K$-sub-Gaussian, having mean zero and a unit second moment. \item The entries of $W$ are either real with a density bounded from above by $\rho_0,$ or entries have independent real and imaginary parts and both have distribution densities bounded from above by $\rho_0$. 
    \end{enumerate}

    Then for any given $R>1$ and $\kappa\in(0,1]$, for any given $z\in\mathbb{C}$ that satisfies $$|z|>\max(\sigma^*n^{2\kappa},\frac{\sigma}{R}),$$ we have that whenever $n$ is sufficiently large (depending on $K,R,\rho_0,\kappa$), with probability at least $1-2n^{-2}$,
    $$
s_{min}(A-zI_n)\geq|z|\exp\left(-R^2n^{3\kappa}(\frac{\sqrt{n}\sigma^*}{\sigma})^2\right).
    $$
\end{theorem}

For dense square matrices with centered i.i.d. entries, polynomial lower bounds on the least singular value are available under suitable moment and nondegeneracy assumptions; see \cite{MR4680362}. For general variance profiles, we use the exponential lower bound of Theorem \ref{tikhoroveheorem}. Combined with the singular-value counting estimates of Section \ref{section33}, it yields the sufficient bandwidth conditions $b_n\geq n^{5/6+\epsilon}$ for Gaussian entries and $b_n\geq n^{8/9+\epsilon}$ for non-Gaussian entries under the respective assumptions of Theorem \ref{circulargeneral}. The conditional range obtainable from a stronger counting estimate is discussed in Section \ref{mainideas}.

\subsection{Squared singular-value measures and Kolmogorov distance}

We introduce the notation for the polynomial comparison of squared singular-value measures in Theorem \ref{theorem5.11}, proved in Section \ref{section33} from the Gaussian and non-Gaussian resolvent estimates. An earlier use of such a comparison for block band matrices appears in \cite[Section 4.2]{jain2021circular}.

Consider the following empirical squared singular-value measures of $X_z:=X-zI$ and $G_z:=G-zI$, where $G=(g_{ij})$ is an $n\times n$ complex Ginibre matrix. Its entries are $g_{ij}=(u_{ij}+iv_{ij})/\sqrt{2n}$, with all $u_{ij},v_{ij}$ independent standard real Gaussians, so $\mathbb E|g_{ij}|^2=n^{-1}$ and $\mathbb E g_{ij}^2=0$. We use this complex Ginibre convention for the reference matrix throughout:
$$
\nu_{X_z}(\cdot)=\frac{1}{n}\sum_{i=1}^n \delta_{\sigma_i^2(X_z)}(\cdot),
$$

$$
\nu_{G_z}(\cdot)=\frac{1}{n}\sum_{i=1}^n \delta_{\sigma_i^2(G_z)}(\cdot).
$$

For two probability measures $\mu,\nu$ on $\mathbb R$, we define their Kolmogorov distance by
\[
\|\mu-\nu\|_{(-\infty,\infty)}:=\sup_{x\in\mathbb R}\left|\mu((-\infty,x])-\nu((-\infty,x])\right|.
\]
When both measures are supported on $[0,\infty)$, this specializes to
\begin{equation}
    \|\mu-\nu\|_{[0,\infty)}:=\sup_{x\geq 0} \left|\mu([0,x])-\nu([0,x])\right|.
\end{equation}

\subsection{A linearized model}

We also need a least singular value result similar to \cite{jain2021circular}, Theorem 2.1 for linearization of product matrices. This will be used in the small-product regime of Proposition \ref{productsmallm}.

\begin{theorem}\label{theorem2.2whatsapp} Let $X_1,\cdots,X_{m_n}$ be independent $n\times n$ random matrices with i.i.d. entries having distribution $\frac{1}{\sqrt{n}}\xi$. Let $m_n$ be an $n$-dependent positive integer satisfying $m_n\leq n$. Assume that $\mathbb{E}[\xi]=0$ and $\mathbb{E}[|\xi|^2]=1$. Moreover, suppose either that $\xi$ is real with a density bounded by $L>0$, or that its real and imaginary parts are independent and at least one has a density bounded by $L>0$.

Consider the following $(nm_n)\times (nm_n)$ matrix 
\begin{equation}\label{welinearizeit}
\mathcal{L}=\begin{pmatrix}
&X_2&&&\\&&X_3&&\\&&&\ddots&\\&&&&X_{m_n}\\X_1&&&&
\end{pmatrix}.\end{equation}
For $m_n=1$, we set $\mathcal L=X_1$.
    Then for any fixed $K'>0$, for any $z\neq 0$ with $|z|\leq K'$ we have
    $$
\mathbb{P}(s_{min}(\mathcal{L}-zI_{nm_n})\leq \min(|z|,1/|z|)^{2m_n}n^{-25m_n})\leq \frac{C_\xi}{\sqrt{n}}
    $$  for a constant $C_\xi>0$ depending only on $K'$ and $L$.

\end{theorem}

Theorem \ref{theorem2.2whatsapp} is proven in Section \ref{section5}. 
\begin{remark} The proof of Theorem \ref{theorem2.2whatsapp} uses the density assumption to bound a scalar linear combination after conditioning on the other rows. Extending this estimate to discrete entry distributions is left as an open problem.
For a fixed number of factors, polynomial lower bounds for the shifted cyclic linearization were obtained in \cite{MR2861673}. For fixed $m_n=m$ and fixed $z\neq0$, sharper quantitative estimates were subsequently proved in \cite{MR4112718} under subgaussian assumptions and in \cite{MR4303893} under finite fourth moment assumptions. The constants in these estimates depend on the fixed number of factors.
\end{remark}

\section{Resolvent comparison, singular-value counting and weak delocalization}\label{section33}

Unless otherwise stated, we assume that $X$ is the general inhomogeneous matrix model in Definition \ref{generalmodel}.

We use free probability and concentration inequalities from \cite{bandeira2023matrix} to derive estimates on Stieltjes transforms. We first work with Gaussian entries. 
For any $n\times n$ matrix $A\in \mathcal{M}_{n}(\mathbb{C})$ with Gaussian entries, $A$ can be written in the form 

$$
A:=A_0+\sum_{i=1}^m A_ig_i
$$ where $g_1,\cdots,g_m$ are i.i.d. standard real Gaussian variables of mean 0, variance one and $A_0,A_1,\cdots,A_m$ are fixed matrices in $\mathcal{M}_{n}(\mathbb{C})$. To this model $A$ we can introduce a model in free probability 
$$
A_{\text{free}}:=A_0\otimes\mathbf{1}+\sum_{i=1}^m A_is_i,
$$ where $s_1,\cdots,s_m$
are a free standard semicircular family in a tracial $C^*$-probability space $(\mathcal{A},\tau)$ (see Section 4.1 of \cite{bandeira2023matrix} for more details).

We consider the dilation matrix of $X_z$.
Define for any $z\in\mathbb{C}$ the dilation matrix \begin{equation}\label{dilationmatrix}\mathcal{Y}_z:=\begin{pmatrix}0&X-zI_n\\X^*-\bar{z}I_n&0\end{pmatrix}.\end{equation} The Stieltjes transform of $\mathcal{Y}_z$ is defined as: for any $\eta\in\mathbb{C}_+:=\{z\in\mathbb{C}:\Im z>0\}$,
$$ \mathcal{G}_z(\eta):=(\mathcal{Y}_z-\eta I_{2n})^{-1},\quad m_z(\eta):=\operatorname{tr}_{2n}\mathcal{G}_z(\eta)=\frac{1}{2n}\operatorname{Tr}\mathcal{G}_z(\eta),
$$ where $\operatorname{tr}_n$ denotes the normalized trace: for a $n\times n$ matrix $M$, $\operatorname{tr}_nM=\frac{1}{n}\operatorname{Tr}M$.

For the Gaussian matrix $X$, let $X^o$ be obtained by replacing each diagonal entry $x_{ii}$ with a centered circular complex Gaussian variable $r_i$ of variance $\mathbb E|x_{ii}|^2$. The variables $r_i$ are mutually independent and independent of $X$, and satisfy $\mathbb E r_i^2=0$. Only the diagonal entries are replaced. Likewise we define $\mathcal{Y}_z^o$ which is $\mathcal{Y}_z$ with $X$ replaced by $X^o$, and define analogously the Green function $\mathcal{G}_{z}^o(\eta)$ and its normalized trace $m_z^o(\eta)$.

The eigenvalues of $\mathcal{Y}_z$ are the signed singular values of $X_z$. We will bound the number of eigenvalues of $\mathcal{Y}_z$ near zero, thereby obtaining polynomial lower bounds for all but a controlled number of singular values of $X_z$.

We also introduce the following \text{free} probability analogues of $\mathcal{Y}_z^o$ and $m_z^o(\eta)$. Define
$$\mathcal{Y}^o_{z,\text{free}}:=\begin{pmatrix}
    0&{X^o}_{\text{free}}-zI_n\\({X^o}_{\text{free}})^*-\bar{z}I_n&0
\end{pmatrix},
$$
$$
\mathcal{G}_{z,\text{free}}^o(\eta):=(\operatorname{id}_{2n}\otimes\tau)\left[(\mathcal{Y}_{z,\text{free}}^o-\eta I_{2n})^{-1}\right],\quad m_{z,\text{free}}^o(\eta):=\operatorname{tr}_{2n}\mathcal{G}_{z,\text{free}}^o(\eta)=\frac{1}{2n}\operatorname{Tr}\mathcal{G}_{z,\text{free}}^o(\eta).
$$

Since the matrix $X^o$ has doubly stochastic variance profile, we show that $m_{z,\text{free}}^o(\eta)$ has a deterministic expression and is independent of the underlying graph chosen. These computations have already been done in \cite{han2024outliers}, Section 3.1 and we recall them here. 
By \cite{haagerup2005new}, equation 1.5 and \cite{han2024outliers}, equation 3.1, the  free probability Stieltjes transform $\mathcal{G}_{z,\text{free}}^o(\eta)$ solves the following self consistency equation
\begin{equation}\label{matrixdysomequationsfag}
\mathbb{E}[\mathcal{Y}_0^o\mathcal{G}_{z,\text{free}}^o(\eta)\mathcal{Y}_0^o]+ (\mathcal G_{z,\text{free}}^o(\eta))^{-1}+\begin{pmatrix} \eta I_n& z I_n\\\bar{z}I_n& \eta I_n
\end{pmatrix}=0.\end{equation} This self consistency equation uniquely characterizes $\mathcal{G}_{z,\text{free}}^o(\eta)$ as by \cite{helton2007operator}, Theorem 2.1, for any $\eta\in\mathbb{C}_+$ there is a unique solution $\mathcal{G}_{z,\text{free}}^o(\eta)$ to this equation such that $\mathcal{G}_{z,\text{free}}^o(\eta)$ has a positive imaginary part. Also note that $\mathbb{E}[(X^o)^2]=0$ and $\mathbb{E}[((X^o)^*)^2]=0$ since we have assumed that the diagonal entries of $X^o$ are circular Gaussians and thus have zero pseudo-variance. Because of this, a standard computation as in \cite{han2024outliers} or \cite{han2025circular2} shows that the unique solution to \eqref{matrixdysomequationsfag} with positive imaginary part is given by 
\begin{equation}\label{definitiongzeta}
\mathcal{G}_{z,\text{free}}^o(\eta)=\begin{pmatrix}
    a(z,\eta) I_n& b(z,\eta)I_n\\ d(z,\eta)I_n&c(z,\eta)I_n
\end{pmatrix},\quad m_{z,\text{free}}^o(\eta)=\frac{1}{2}(a(z,\eta)+c(z,\eta)),
\end{equation} where $a(z,\eta)$, $b(z,\eta)$, $c(z,\eta)$ and $d(z,\eta)$ are scalar functions depending on $z$ and $\eta$. Moreover, $a$, $b$, $c$, $d$ solve the scalar equations (see \cite{han2024outliers}, equation (3.4))
\begin{equation} \label{agageggwgewe}
\begin{pmatrix}
    c&0\\0&a
\end{pmatrix}
+
    \frac{1}{ac-bd}\begin{pmatrix}
        c&-b\\-d&a
    \end{pmatrix}+\begin{pmatrix}
        \eta&z\\\bar{z}&\eta
    \end{pmatrix}=0.
\end{equation}
When we specialize to $\eta=iv,v>0$, we have the further symmetry
$$
d=\bar{b}.
$$

From this computation we see that $\mathcal{G}_{z,\text{free}}^o(\eta)$, and hence $m_{z,\text{free}}^o(\eta)$, are independent of the specific variance profile for the inhomogeneous non-Hermitian matrix $X$, and in particular the free Stieltjes transform $m_{z,\text{free}}^o(\eta)$ associated with $X_z^o$ coincides with that associated with $G_z$, where $G$ has i.i.d. centered circular complex Gaussian entries of variance $n^{-1}$.

The next task is to correctly identify the rate of convergence of the Stieltjes transform $m_{z}(\eta)$ towards the deterministic limit $m_{z,\text{free}}^o(\eta)$. 

 We need the following perturbation estimate between $\mathcal Y_z$ and $\mathcal Y_z^o$:

\begin{lemma}
The perturbation formula \eqref{resolventperturbartion} holds for the resolvent between $X$ and $X^o$.    
\end{lemma}
\begin{proof}
    Set 
$$
D=X-X^o=\operatorname{diag}(x_{11}-r_1,\cdots,x_{nn}-r_n)
$$where $r_i$ is a circular Gaussian of the same variance as $x_{ii}$. Then the Hermitization difference is 
$$
\mathcal{Y}_z-\mathcal{Y}_z^o=\begin{bmatrix}
    0&D\\D^*&0
\end{bmatrix}:=\Delta.
$$ Then we have 
$$
\|\Delta\|=\|D\|=\max_i|x_{ii}-r_i|.
$$ Let $\eta\in\mathbb{C}_+$ with $v=\Im\eta>0$, the resolvent identity gives 
$$
\mathcal{G}_z-\mathcal{G}_z^o=-\mathcal{G}_z\Delta \mathcal{G}_z^o.
$$ Since $x_{ii}-r_i$ is centered Gaussian with second moment at most $2/{b_n}$, we apply Gaussian concentration and take a union bound over $i=1,\cdots,n$ to get that, for any $D>0$,
$$
\max_i|x_{ii}-r_i|\leq C_D\sqrt{\frac{\log n}{b_n}}
$$with probability at least $1-n^{-D}$. We also use $\|\mathcal{G}_z(\eta)\|\leq\frac{1}{v}$ and $\|\mathcal{G}_z^o(\eta)\|\leq\frac{1}{v}$. Combining these, we get that for any $D>0$, with probability $1-n^{-D}$, 
\begin{equation}\label{resolventperturbartion}
|m_z(\eta)-m_z^o(\eta)|\leq \frac{C_D\sqrt{\log n}}{\sqrt{b_n}(\Im\eta)^2},\quad \|\mathcal{G}_z(\eta)-\mathcal{G}_z^o(\eta)\|\leq \frac{C_D\sqrt{\log n}}{\sqrt{b_n}(\Im\eta)^2}.
\end{equation}
\end{proof}
We first consider the case where $X$ has independent Gaussian entries. The following proposition follows from Theorem 2.8 and Lemma 4.13 of \cite{bandeira2023matrix}.
\begin{Proposition}\label{proposition3.1212}
    Assume that $X=(x_{ij})$ is a $n\times n$ random matrix. We assume that $X$ has independent centered Gaussian entries and a doubly stochastic variance profile $\mathbb{E}[XX^*]=\mathbb{E}[X^*X]=\mathbf{1}$. Let 
    $$b_n^{-1}:=\sup_{(i,j)\in[n]^2}\mathbb{E}[|x_{ij}|^2].$$
  Then for every fixed $z\in\mathbb C$ and every deterministic $\eta=\eta_n\in\mathbb{C}_+$, with probability at least $1-n^{-100}$, we have
    \begin{equation}\label{convergencestieltjes}
|\operatorname{tr}_{2n}\mathcal{G}_z(\eta)-\operatorname{tr}_{2n}\mathcal{G}_{z,\text{free}}^o(\eta)|\leq \frac{C}{b_n|\Im\eta|^5}+\frac{C(\log n)^3}{\sqrt{b_n}|\Im\eta|^2}
    \end{equation} for some universal constant $C>0$.

In particular, for every fixed $z$ with $|z|\leq3$, every $\mathbf{c}>0$ and every deterministic $\eta=\eta_n\in\mathbb{C}_+$ with $|\eta|\leq10$ and $\Im\eta\geq (b_n)^{-1/5}n^\mathbf{c}$, with probability at least $1-n^{-100}$,
\begin{equation}\label{boundonoperatornorm}
|\operatorname{tr}_{2n}\mathcal{G}_z(\eta)|\leq C_1
\end{equation} where $C_1$ is a universal constant, uniform over the indicated range of $z$ and $\eta$. These probability statements are pointwise in $z$ and $\eta$.
\end{Proposition}

\begin{proof} We take the comparison to free object with respect to the version $X^o$ with circular Gaussian diagonal entries.
 By \cite{bandeira2023matrix}, Theorem 2.8 we have 
    \begin{equation}\label{freeproba}
|\mathbb{E}m_z^o(\eta)-m_{z,\text{free}}^o(\eta)|\leq\frac{C}{b_n(\Im\eta)^5}.
    \end{equation}
To obtain the concentration estimate, set $v=\Im\eta>0$. For any $2n\times2n$ Hermitian matrices $H,K$, the resolvent identity and Hilbert--Schmidt Cauchy--Schwarz inequality give
\[
\left|\operatorname{tr}_{2n}(H-\eta I_{2n})^{-1}
-\operatorname{tr}_{2n}(K-\eta I_{2n})^{-1}\right|
\leq\frac{\|H-K\|_{\mathrm{HS}}}{\sqrt{2n}\,v^2}.
\]
For each entry of $X^o$, the covariance matrix of its real and imaginary parts has largest eigenvalue at most $\mathbb{E}|x_{ij}^o|^2\leq b_n^{-1}$. Entry independence and the factor $\sqrt{2}$ in the Hilbert--Schmidt norm of the dilation therefore show that $m_z^o(\eta)$, viewed as a function of independent standard real Gaussians, is Lipschitz with constant at most $(nb_n)^{-1/2}v^{-2}$ in Euclidean norm. Applying the Gaussian concentration inequality in \cite{bandeira2023matrix}, Lemma 4.13, separately to its real and imaginary parts gives, for any $D\geq1$, with probability at least $1-n^{-D}$,
\[
|m_z^o(\eta)-\mathbb{E}m_z^o(\eta)|\leq \frac{C_D\sqrt{\log n}}{\sqrt{nb_n}(\Im\eta)^2}.
\]
The perturbation formula \eqref{resolventperturbartion} also gives high probability upper bound of $|m_z(\eta)-m_z^o(\eta)|$. The claim \eqref{convergencestieltjes} then follows.
 The bound \eqref{boundonoperatornorm} for each fixed $\eta$ in this regime is then a direct consequence of the uniform upper bound on $\mathcal{G}_{z,\text{free}}^o$ that can be read off from the solution of the cubic equation \eqref{agageggwgewe}, see for example \cite{girko2012theory}, \cite{MR3770875}. For each fixed $z$, a net argument makes the bound simultaneous in $\eta$; the details and probability budget are given in the proof of Corollary \ref{corollary3.45}.
\end{proof}

Combining the comparison estimate \eqref{convergencestieltjes} with the global free-transform bound $|m_{z,\mathrm{free}}^o(\eta)|\leq1$, proved below, yields an eigenvalue-counting bound for every fixed $z\in\mathbb C$.

\begin{corollary}\label{corollary3.45}
Let $X$ satisfy the assumptions in Proposition \ref{proposition3.1212}. Then for every fixed $z\in\mathbb{C}$ and every $\mathbf{c}>0$, there is a universal constant $C_2$ such that, with probability at least $1-n^{-10}$, the following estimate holds simultaneously for all intervals $I\subset[-5,5]$ with length $$|I|\geq (b_n)^{-1/5}n^\mathbf{c}:$$
$$
\#\{\lambda\in I:\text{$\lambda$ is an eigenvalue of $\mathcal{Y}_z$} \}\leq C_2n|I|.
$$

\end{corollary}

\begin{proof}
We first establish a global bound for the free Stieltjes transform. Subtracting the two diagonal equations in \eqref{agageggwgewe} shows that $a=c$: otherwise $ac-bd=-1$ and then $\eta=0$, which is impossible in $\mathbb C_+$. Write $m=a=c=m_{z,\mathrm{free}}^o(\eta)$ and $r=|z|$. Inverting the resulting $2\times2$ matrix gives
\[
m=-\frac12\left(\frac{1}{\eta+m+r}+\frac{1}{\eta+m-r}\right).
\]
For $v=\Im\eta>0$ and $y=\Im m>0$, taking imaginary parts and applying Cauchy--Schwarz yield
\[
\begin{aligned}
y&=\frac{v+y}{2}\left(\frac{1}{|\eta+m+r|^2}+\frac{1}{|\eta+m-r|^2}\right),\\
|m|^2&\leq\frac12\left(\frac{1}{|\eta+m+r|^2}+\frac{1}{|\eta+m-r|^2}\right)
=\frac{y}{v+y}\leq1.
\end{aligned}
\]
Thus $|m_{z,\mathrm{free}}^o(\eta)|\leq1$ for every $z\in\mathbb C$ and $\eta\in\mathbb C_+$.

Now fix $z\in\mathbb C$ and $\mathbf c>0$, and set $h=b_n^{-1/5}n^{\mathbf c}$. If $h\geq1$, the result follows from the total eigenvalue count $2n$. Otherwise consider
\[
\mathcal D_h=\{u+iv:-5\leq u\leq5,\ h\leq v\leq10\}.
\]
On this rectangle the right hand side of \eqref{convergencestieltjes} is at most
\[
Cn^{-5\mathbf c}+C(\log n)^3b_n^{-1/10}n^{-2\mathbf c}=o(1).
\]
Apply \eqref{convergencestieltjes} on an $n^{-2}$-net of $\mathcal D_h$, containing $O(n^4)$ points. The union failure probability is $O(n^{-96})$, since the pointwise estimate fails with probability at most $n^{-100}$. Both Stieltjes transforms are $h^{-2}$-Lipschitz there, and $b_n\leq n$ gives $h\geq n^{-1/5}$. Hence the interpolation error is $O(n^{-2}h^{-2})=O(n^{-8/5})$. The union bound and the global free bound therefore imply, for all sufficiently large $n$,
\[
\sup_{\eta\in\mathcal D_h}|m_z(\eta)|\leq2
\]
with probability at least $1-n^{-10}$.

On this single event, let $I\subset[-5,5]$ have length $\ell\geq h$ and midpoint $u$. Since $\ell\leq10$, we have $u+i\ell\in\mathcal D_h$. Every eigenvalue in $I$ contributes at least $(2\ell)^{-1}$ to the unnormalized imaginary resolvent trace, so
\[
\#\{\lambda\in I:\lambda\text{ is an eigenvalue of }\mathcal Y_z\}
\leq4n\ell\,\Im m_z(u+i\ell)\leq8n\ell.
\]
This proves the claim simultaneously for all such intervals, with a universal constant $C_2$.
\end{proof}

We next treat general entry distributions under a finite third moment assumption. We derive the resolvent comparison by combining the trace universality argument in \cite{brailovskaya2022universality}, Remark 6.13, with Gaussian--free comparison, scalar concentration and a correction of the diagonal pseudovariances. 

\begin{Proposition}\label{newproposition3.2}
    Assume that $X=(x_{ij})$ is a $n\times n$ random matrix with independent entries, having mean zero and a doubly stochastic variance profile $\mathbb{E}[XX^*]=\mathbb{E}[X^*X]=\mathbf{1}$. For each $(i,j)\in[n]^2$  $x_{ij}$ is distributed as $b_{ij}\xi$ where $\xi$ has mean zero, variance one and finite third moment. Set $$b_n^{-1}=\sup_{(i,j)\in[n]^2}b_{ij}^2$$ and assume that $\log_n b_n>\epsilon>0$ for some $\epsilon>0$. Fix $z\in\mathbb C$ and a deterministic $\eta=\eta_n\in\mathbb C_+$. Assume that for some (small) $\mathbf{c}'>0$,
    $$|\Im\eta|^8b_n\geq n^{\mathbf{c}'}.$$Set $\mathbf{c}''=\min\{\mathbf{c}'/8,1/8\}$, independently of $\eta$. Then, for all sufficiently large $n$, the following estimate holds with probability at least $1-n^{-10}$:
    \begin{equation}\label{convergencestieltjesnon}
|\operatorname{tr}_{2n}\mathcal{G}_z(\eta)-\operatorname{tr}_{2n}\mathcal{G}^o_{z,\text{free}}(\eta)|\leq n^{-\mathbf{c}''},
    \end{equation}
    Moreover, for every fixed $z$ with $|z|\leq3$ and every $\mathbf c'>0$, with probability at least $1-n^{-10}$,
    \begin{equation}\label{estimate272}
|\operatorname{tr}_{2n}\mathcal{G}_z(\eta)|\leq C_1
    \end{equation} holds simultaneously for all $\eta\in\mathbb{C}_+$ satisfying $|\eta|\leq 5$ and $\Im\eta\geq (b_n)^{-1/8}n^{\mathbf{c}'}$, where $C_1>0$ is universal. The event may depend on $z$, while $C_1$ does not.
\end{Proposition}
For this result we do not need to assume higher moments on $x_{ij}$: only finite third moment is enough. For the general model, we use the subgaussian assumption in the operator norm and least singular value estimates. In the small-product regime of Proposition \ref{productsmallm}, the block structure instead allows us to use Lemma \ref{linearizationgoodevent} and Theorem \ref{theorem2.2whatsapp} under the stated moment assumptions.
\begin{proof} For a non-Gaussian $X$, let $X^G=(g_{ij})$ have independent Gaussian entries with matching first and second real moments (that is, $\mathbb{E}g_{ij}=0,\mathbb{E}|g_{ij}|^2=\mathbb{E}|x_{ij}|^2,\mathbb{E}g_{ij}^2=\mathbb{E}x_{ij}^2$). Replacing the diagonal entries of $X^G$ by independent circular Gaussians of matching variances gives $X^{G,o}$. Its free Stieltjes transform is the uniquely determined $m_{z,\text{free}}^o(\eta)$ from \eqref{definitiongzeta}--\eqref{agageggwgewe}.

The claimed estimate follows from the following chain of inequalities: for any $D>0$ we can find $C_D>0$ such that with probability $1-n^{-D}$, we have

\begin{equation}\begin{aligned}\label{lines495}
|m_z(\eta)-m_{z,\text{free}}^o(\eta)|\leq\frac{C}{\sqrt{b_n}v^4}+\frac{C}{b_nv^5}+\frac{C_D\sqrt{\log n}}{\sqrt{n}v}+\frac{C\sqrt{\log n}}{\sqrt{b_n}v^2},
\end{aligned}\end{equation}where $v=\Im\eta>0$.

We explain the four terms on the right hand side separately.

(1)Trace universality:
$$
|\mathbb{E}m_z(\eta)-\mathbb{E}m_z^G(\eta)|\leq\frac{C}{\sqrt{b_n}v^4}.
$$We apply the trace universality argument in \cite{brailovskaya2022universality}, Remark 6.13, to the Hermitian dilation. Write $\mathcal Y_z=\mathbb E\mathcal Y_z+\sum_{i,j}Z_{ij}$, where $Z_{ij}=x_{ij}E_{i,n+j}+\overline{x_{ij}}E_{n+j,i}$ and $E_{ab}$ are the $2n\times2n$ matrix units. The increments $Z_{ij}$ are independent, centered and Hermitian, and the deterministic shift $\mathbb E\mathcal Y_z$ is held fixed. This gives
\begin{equation}\label{onlines505}
|\mathbb{E}m_z(\eta)-\mathbb{E}m_z^G(\eta)|\lesssim \frac{1}{(\Im\eta)^4}\sum_{i,j}\mathbb{E}[\operatorname{tr}_{2n}|Z_{ij}|^3],
\end{equation}

For these entry increments,
$$
\sum_{i,j}\mathbb{E}\operatorname{tr}_{2n}|Z_{ij}|^3=\frac{\mathbb{E}|\xi|^3}{n}\sum_{i,j}b_{ij}^3\leq\frac{C_\xi}{\sqrt{b_n}}.
$$This is the only place where the finite third moment is used. The inequality \eqref{onlines505} was originally stated for bounded summands, but the proof extends verbatim to the unbounded case. The proof follows a Lindeberg replacement argument and Taylor expansion. First, we replace $Z_{ij}$ by matching Gaussians $G_{ij}$ one summand at a time. Then we Taylor expand $F(A)=\operatorname{tr}_{2n}(\eta-A)^{-1}$ up to second order. Then the zeroth, first and second order terms cancel since they have matching first and second real moments. Finally we bound the third order remainder via 
$$|D^3F(A)[H,H,H]|\leq 6v^{-4}\operatorname{tr}_{2n}|H|^3,
$$
and sum over $i,j$.

(2)The Gaussian to free object comparison:
$$
|\mathbb{E}m_z^{G,o}(\eta)-m_{z,\text{free}}^o(\eta)|\leq \frac{C}{b_nv^5},
$$
this already follows from \cite{bandeira2023matrix}, Theorem 2.8 as in the Gaussian case Proposition \ref{proposition3.1212}.

(3) Scalar concentration: with probability $1-n^{-D}$ we have
$$
|m_z(\eta)-\mathbb{E}m_z(\eta)|\leq \frac{C_D\sqrt{\log n}}{\sqrt{n}v}.
$$
To prove this, we note that replacing one row of $X$ changes $\mathcal{Y}_z$ by rank at most 2, so that it changes its normalized Stieltjes transform by an amount at most $C/nv$. Then McDiarmid concentration inequality gives
$$
\mathbb{P}(|m_z(\eta)-\mathbb{E}m_z(\eta)|\geq t)\leq 4e^{-cnv^2t^2}.
$$
No moment assumption on $X$ is needed at this step.

(4) Diagonal pseudovariance correction:
$$
|\mathbb{E}m_z^G-\mathbb{E}m_z^{G,o}|\leq\frac{C\sqrt{\log n}}{\sqrt{b_n}v^2}.
$$This follows from \eqref{resolventperturbartion}.

We now make the decay exponent independent of the spectral parameter. The variance normalization implies $1\leq b_n\leq n$. If $b_nv^8\geq n^{\mathbf{c}'}$, then the four terms in \eqref{lines495} satisfy
\[
\begin{aligned}
\frac{C}{\sqrt{b_n}v^4}&\leq Cn^{-\mathbf{c}'/2},\\
\frac{C}{b_nv^5}&\leq Cb_n^{-3/8}n^{-5\mathbf{c}'/8},\\
\frac{C_D\sqrt{\log n}}{\sqrt{n}v}&\leq C_D\sqrt{\log n}\,n^{-3/8-\mathbf{c}'/8},\\
\frac{C\sqrt{\log n}}{\sqrt{b_n}v^2}&\leq C\sqrt{\log n}\,b_n^{-1/4}n^{-\mathbf{c}'/4}.
\end{aligned}
\]
Their sum is at most $n^{-\mathbf{c}''}$ for $\mathbf{c}''=\min\{\mathbf{c}'/8,1/8\}$ and all sufficiently large $n$. The required lower bound on $n$ may depend on $\mathbf{c}'$, the atom distribution and $D$, but is independent of $\eta$ and $z$. Taking $D=10$ proves \eqref{convergencestieltjesnon}, also for any deterministic sequence $\eta=\eta_n\in\mathbb C_+$ satisfying the displayed scale condition.

For the uniform bound \eqref{estimate272}, fix $z$ and set $h=b_n^{-1/8}n^{\mathbf{c}'}$. If $h\geq1$, the deterministic resolvent bound gives $|m_z(\eta)|\leq h^{-1}\leq1$. If $h<1$, take an $n^{-2}$-net of $\{\eta\in\mathbb C_+:|\eta|\leq5,\ \Im\eta\geq h\}$, with $O(n^4)$ points. Apply \eqref{lines495} with $D=20$ at these points and use the preceding calculation with $8\mathbf{c}'$ in place of $\mathbf{c}'$. The comparison error tends to zero uniformly on the net, with total failure probability $O(n^{-16})$. Both Stieltjes transforms are $h^{-2}$-Lipschitz on this region, so the interpolation error is $O(n^{-2}h^{-2})=O(n^{-7/4})$. The global free-transform bound $|m^o_{z,\mathrm{free}}|\leq1$, proved in Corollary \ref{corollary3.45}, therefore gives \eqref{estimate272} with $C_1=2$ and probability at least $1-n^{-10}$ for all sufficiently large $n$. This constant is uniform in $|z|\leq3$.
\end{proof}

The same counting argument applies to the non-Gaussian model at the following scale:

\begin{corollary}\label{corollary3.45news}
Let $X$ satisfy the assumptions in Proposition \ref{newproposition3.2}. Then for every fixed $z\in\mathbb{C}$ and every $\mathbf{c}>0$, there is a universal constant $C_2$ such that, with probability at least $1-n^{-10}$, the following estimate holds simultaneously for all intervals $I\subset[-5,5]$ with length $$|I|\geq (b_n)^{-1/8}n^\mathbf{c}:$$
$$
\#\{\lambda\in I:\text{$\lambda$ is an eigenvalue of $\mathcal{Y}_z$} \}\leq C_2n|I|.
$$

\end{corollary}

\begin{proof}
Fix $z\in\mathbb C$ and $\mathbf c>0$, and set $h=b_n^{-1/8}n^{\mathbf c}$. The case $h\geq1$ is again immediate. On the rectangle $\mathcal D_h$ defined in the proof of Corollary \ref{corollary3.45}, the four terms in \eqref{lines495}, with $D=20$, are bounded by
\[
\begin{aligned}
&Cn^{-4\mathbf c}+Cb_n^{-3/8}n^{-5\mathbf c}\\
&\quad+C_{20}\sqrt{\log n}\,n^{-3/8-\mathbf c}
+C\sqrt{\log n}\,b_n^{-1/4}n^{-2\mathbf c}=o(1),
\end{aligned}
\]
where we used $b_n\leq n$. Apply this estimate on an $n^{-2}$-net of $\mathcal D_h$ with $O(n^4)$ points. The union failure probability is $O(n^{-16})$, and, since $h\geq n^{-1/8}$, the interpolation error is $O(n^{-2}h^{-2})=O(n^{-7/4})$. The global bound $|m_{z,\mathrm{free}}^o|\leq1$ proved above therefore gives $\sup_{\eta\in\mathcal D_h}|m_z(\eta)|\leq2$ with probability at least $1-n^{-10}$ for all sufficiently large $n$. The same Poisson-kernel estimate as in Corollary \ref{corollary3.45} now proves the claim for every eligible interval on this event.
\end{proof}

\subsection{Kolmogorov comparison of squared singular-value measures}
We now use the already derived Green function convergence estimates in Proposition \ref{proposition3.1212} and \ref{newproposition3.2} to obtain the following polynomial convergence rate in Kolmogorov distance:

\begin{theorem}\label{theorem5.11}
    Let $X$ be defined as in Definition \ref{generalmodel} with subgaussian entries and $b_n\geq n^\epsilon$ for some small $\epsilon>0$. Then for any fixed $z\in\mathbb{C}$ we can find $\zeta>0$ such that with probability $1-o(1)$, we have
\begin{equation}\label{convergencesincigyag}
    \|\nu_{X_z}(\cdot)-\nu_{G_z}(\cdot)\|_{[0,\infty)}=O(n^{-\zeta}).
\end{equation}
\end{theorem}

\begin{proof}[\proofname\ of Theorem \ref{theorem5.11}]
    Let $\widetilde{G}_z$ denote the dilation matrix, where $G$ is complex Ginibre: 
    $$
\widetilde{G}_z:=\begin{pmatrix}0&G-zI\\G^*-\bar{z}I&0\end{pmatrix}. 
    $$ Observing that $\lambda$ is an eigenvalue of $X_zX_z^*$ (resp. $G_zG_z^*$) if and only if $\pm\sqrt{\lambda}$ are both eigenvalues of $\mathcal{Y}_z$ (resp. $\widetilde{G}_z$), we see that to prove \eqref{convergencesincigyag}, it suffices to prove
    \begin{equation}\label{convergencereduced}
\|\mu_{\mathcal{Y}_z}(\cdot)-\mu_{\widetilde{G}_z}(\cdot)\|_{(-\infty,\infty)}=O(n^{-\zeta}),
    \end{equation} where we recall that $\mu_X$ denotes the empirical eigenvalue distribution (it differs from $\nu_X$ by considering the eigenvalues instead of considering squared singular values) because we have $$
\|\nu_{X_z}(\cdot)-\nu_{G_z}(\cdot)\|_{[0,\infty)}\leq 2\|\mu_{\mathcal Y_z}(\cdot)-\mu_{\widetilde{G}_z}(\cdot)\|_{(-\infty,\infty)}.
    $$
    The benefit of this reduction is that both $\mathcal{Y}_z$ and $\widetilde{G}_z$ have independent entries modulo symmetry. To prove \eqref{convergencereduced}, we show the Stieltjes transform of $\mathcal{Y}_z$ converges to the Stieltjes transform of $\widetilde{G}_z$ at a polynomial rate. Recall that $m_z(\eta):=\operatorname{tr}_{2n}\mathcal{G}_z(\eta)$ is the Stieltjes transform of $\mathcal{Y}_z$, and $m_{z,\text{free}}^o(\eta)$ is  the Stieltjes transform of $\mathcal{Y}^o_{z,\text{free}}$. Here the expression $m^o_{z,\text{free}}(\eta)$ has been explicitly computed in \eqref{definitiongzeta}, \eqref{agageggwgewe}.

We first record the operator-norm input used for both cases. There is a fixed $K_\xi\geq3$, depending only on the subgaussian parameter of the atom, such that, for every fixed $D>0$ and all sufficiently large $n$,
\begin{equation}\label{generalnormbound}
\mathbb P(\|X\|>K_\xi)\leq n^{-D}.
\end{equation}
Indeed, write $X=U+iV$ with real matrices $U,V$, and symmetrize each using an independent copy. By conditional Jensen's inequality, the $2p$-th moments of their operator norms are bounded by those of the symmetrized matrices. For a fixed $L>0$ depending only on the subgaussian parameter, each symmetrized real atom $\theta$ satisfies $\mathbb E\theta^{2q}\leq L^{2q}\mathbb E g^{2q}$ for all integers $q\geq1$, where $g$ is a standard real Gaussian; odd moments vanish. Expanding the even traces of the real symmetric dilations therefore gives Gaussian moment domination. The Gaussian comparison dilation, multiplied by $\sqrt{b_n}$, has dimension $2n$, maximum row variance $b_n$ and maximum coefficient at most one. Thus \cite{bandeira2016sharp}, Proposition 2.1 and Lemma 2.2, together with the $L^{2p}$ triangle inequality, give
\[
\bigl(\mathbb E\|X\|^{2p}\bigr)^{1/(2p)}
\leq 2L(2n)^{1/(2p)}\left[
2\sqrt{\frac{\lceil b_n\rceil+p}{b_n}}
+2\sqrt{\frac{2p}{b_n}}\right],\qquad p\in\mathbb N,\ p\geq2.
\]
Take $p=\lceil(D+2)\log(2n)\rceil$. Since $b_n\geq n^\epsilon$, the right hand side is eventually bounded by a fixed $C_\xi$, independently of $D$. Markov's inequality with $K_\xi\geq eC_\xi$ gives \eqref{generalnormbound}. This argument does not require independence between the real and imaginary parts of an atom. The same argument applies to the Ginibre reference $G$; enlarging $K_\xi$ if necessary, we use \eqref{generalnormbound} for both $X$ and $G$.

    (The Gaussian case) First assume that $X$ has Gaussian entries.
    Let $\mu_z$ denote the common free spectral measure of $\mathcal Y_{z,\mathrm{free}}^o$ and $\widetilde G_{z,\mathrm{free}}$. Its Stieltjes transform is $m_{z,\mathrm{free}}^o(\eta)$, since both free models satisfy \eqref{definitiongzeta}--\eqref{agageggwgewe}. The comparison with the original $X$ is already included in \eqref{convergencestieltjes} through the diagonal perturbation estimate \eqref{resolventperturbartion}. By \cite{MR2663633}, Lemma 3.1 and Remark 3.1, $\mu_z(\cdot)$ has a bounded support and bounded density: for each $x\in\mathbb{R},y\geq 0:$
\begin{equation}\label{boundeddensity}
|\mu_z((-\infty,x+y))-\mu_z((-\infty,x))|\leq y.\end{equation}

    By \eqref{convergencestieltjes}, for each fixed $z\in\mathbb{C}$ and each deterministic $\eta=\eta_n\in\mathbb C_+$, with probability at least $1-n^{-100}$ we have
    \begin{equation}\label{convergencerate12}
|m_z(\eta)-m^o_{z,\text{free}}(\eta)|\lesssim\frac{C}{b_n|\Im \eta|^5}+\frac{C(\log n)^3}{\sqrt{b_n}(\Im\eta)^2}.
    \end{equation}
Fix a sufficiently large $A>0$ and define the region 
$$
\mathcal{D}_A:=\{\eta\in\mathbb{C}:-A\leq\Re \eta\leq A,\quad (b_n)^{-1/6}\leq \Im \eta\leq 1\}.
$$
Using the Lipschitz continuity of $m_z(\eta)$ and $m_{z,\mathrm{free}}^o(\eta)$ in $\eta$, we can upgrade the convergence in \eqref{convergencerate12} to be uniform over $\eta\in\mathcal{D}_A$, with probability at least $1-n^{-10}$. 

In the following take a fixed $K>K_\xi+|z|+2$. By \eqref{generalnormbound}, the empirical Hermitization laws of both $X_z$ and $G_z$ are supported in $[-K,K]$ with probability $1-o(1)$; the common free law is also supported in this interval. Also take some $a>0$ sufficiently large.
Now we use \cite{MR2567175}, Corollary B.15 to derive, for some $A>0$ large depending on $K$ and $a$, (writing $\eta=\theta+i\tau$)
$$\begin{aligned}
&\|\mu_{\mathcal{Y}_z}(\cdot)-\mu_z(\cdot)\|_{(-\infty,\infty)}\\&\leq C\left[
\int_{-A}^A|m_{z}(\eta)-m^o_{z,\text{free}}(\eta)|d\theta+\frac{1}{\tau}\sup_x\int_{|y|\leq 2\tau a}|\mu_{z}((-\infty,x+y])-\mu_{z}((-\infty,x])||dy
\right]
\end{aligned}$$ for some $C$ depending only on $A$, $a$ and $K$. Now we set $\tau=(b_n)^{-1/6}$. To bound the first term on the second line we use \eqref{convergencerate12}, and to bound the second term on the second line use \eqref{boundeddensity}. Since $G$ has circular Gaussian entries, its free probability Stieltjes transform is equal to $m_{z,\text{free}}^o$ as they are both the solution to \eqref{agageggwgewe}. Doing the same computation for $\widetilde{G}_z$ completes the proof of \eqref{convergencereduced}.

(The non-Gaussian case) Set $\tau=b_n^{-1/16}$. We use the explicit finite-third-moment estimate \eqref{lines495}. Apply it with $D=12$ on an $n^{-2}$-net of $[-A,A]$, and use the $\tau^{-2}$ Lipschitz bound for both Stieltjes transforms to extend the comparison to the whole interval. Since $b_n\leq n$ and $b_n\geq n^\epsilon$, this gives, with probability $1-o(1)$,
\[
\begin{aligned}
\sup_{|\theta|\leq A}|m_z(\theta+i\tau)-m^o_{z,\mathrm{free}}(\theta+i\tau)|
&\leq Cb_n^{-1/4}+Cb_n^{-11/16}
+C\sqrt{\log n}\bigl(n^{-7/16}+b_n^{-3/8}\bigr)\\
&=O(b_n^{-1/4}),
\end{aligned}
\]
where the interpolation error $O(n^{-2}\tau^{-2})$ is absorbed into the right hand side. The same estimate holds for the Ginibre reference, whose bandwidth is $n\geq b_n$. On the fixed operator-norm event used above, apply the same compact-support smoothing inequality and \eqref{boundeddensity}. Each Hermitization law is then within $O(b_n^{-1/4}+\tau)=O(b_n^{-1/16})$ of the common free law in Kolmogorov distance. Passing to squared singular values proves \eqref{convergencesincigyag}; in particular, one may take $\zeta=\epsilon/16$.
\end{proof}

\subsection{Weak delocalization}
To obtain eigenvector delocalization, we will prove the following version of diagonal Green function entry estimates.
\begin{Proposition}
    Assume that $X=(x_{ij})$ has independent centered real or complex Gaussian entries, and $\max_{ij}\mathbb{E}|x_{ij}|^2\leq (b_n)^{-1}$ and $\mathbb{E}[XX^*]=\mathbb{E}[X^*X]=I_n$. Assume that $b_n\geq n^\epsilon$ for some $\epsilon>0$. For any fixed (small) $\mathbf{c}>0$, we set 
    $v_n=(b_n)^{-1/5}n^{2\mathbf{c}}.$ Then with probability at least $1-n^{-10}$, we have     \begin{equation}\label{greenfunctionupperbounds}
\sup_{|z|\leq 3}\max_{1\leq k\leq 2n}\|\mathcal G_z(iv_n)_{kk}\|\le C,\end{equation}
    where the constant $C>0$ is independent of $n$. 
\end{Proposition}
\begin{proof}
Fix $D>0$. We prove the estimate with failure probability at most $n^{-D}$. The diagonal entries of $\mathcal G_{z,\mathrm{free}}^o(iv_n)$ equal $m_{z,\mathrm{free}}^o(iv_n)$ and have modulus at most one, by the global bound proved in Corollary \ref{corollary3.45}. By \cite{bandeira2023matrix}, Theorem 2.8,
\[
\|\mathbb E\mathcal G_z^o(iv_n)-\mathcal G_{z,\mathrm{free}}^o(iv_n)\|
\leq\frac{C}{b_nv_n^5}=Cn^{-10\mathbf c}.
\]
The perturbation estimate \eqref{resolventperturbartion} holds uniformly in $z$ on the event
\[
\|X-X^o\|\leq C_{D+2}\sqrt{\frac{\log n}{b_n}},
\]
which has probability at least $1-n^{-D-2}$. On this event its contribution is at most
\[
\frac{C_{D+2}\sqrt{\log n}}{\sqrt{b_n}v_n^2}
=C_{D+2}\sqrt{\log n}\,b_n^{-1/10}n^{-4\mathbf c}=o(1).
\]
For each fixed $z$ and $k$, Gaussian concentration applied separately to the real and imaginary parts, as in \cite{bandeira2023matrix}, Corollary 4.14, gives
\[
\mathbb P\left(\left|(\mathcal G_z^o(iv_n))_{kk}
-\mathbb E(\mathcal G_z^o(iv_n))_{kk}\right|\geq1\right)
\leq4\exp(-c_0b_nv_n^4).
\]
Here $b_nv_n^4=b_n^{1/5}n^{8\mathbf c}\geq n^{\epsilon/5}$. Take an $n^{-2}$-net of the disk $|z|\leq3$, of size $O(n^4)$, and apply a union bound over its points and the $2n$ diagonal entries. The resulting concentration failure probability is at most
\[
Cn^5\exp(-c_0n^{\epsilon/5})\leq n^{-D-2}
\]
for all sufficiently large $n$. Finally, the resolvent identity gives
\[
\left|(\mathcal G_z(iv_n))_{kk}-(\mathcal G_{z'}(iv_n))_{kk}\right|
\leq\frac{|z-z'|}{v_n^2}.
\]
Since $b_n\leq n$ and $v_n\geq n^{-1/5}$, interpolation costs only $O(n^{-2}v_n^{-2})=O(n^{-8/5})$. The free bound, comparison and concentration estimates therefore prove \eqref{greenfunctionupperbounds}, with failure probability at most $2n^{-D-2}\leq n^{-D}$. This also records the stronger confidence estimate needed below.
\end{proof}

Then the weak delocalization statement for the Gaussian model follows:
\begin{proof}[\proofname\ of Theorem \ref{delocalizationtheorem}]
It suffices to prove the claim for $0<\mathbf c<\epsilon/10$, since a bound with a smaller positive exponent implies the bound for any larger exponent. Apply the preceding proposition with its parameter $\mathbf c/2$, so that now
\[
v_n=b_n^{-1/5}n^{\mathbf c}.
\]
Its proof gives \eqref{greenfunctionupperbounds} with probability at least $1-n^{-12}$.

We also justify the operator-norm event. The row and column variance matrices are $I_n$, entry independence gives $\|\operatorname{Cov}(\operatorname{vec}X)\|\leq b_n^{-1}$, and
\[
\sup_{\|u\|_2=\|w\|_2=1}\mathbb E|u^*Xw|^2
=\sup_{\|u\|_2=\|w\|_2=1}\sum_{i,j}\mathbb E|x_{ij}|^2|u_i|^2|w_j|^2
\leq b_n^{-1}.
\]
Thus \cite{bandeira2023matrix}, Corollary 2.2 and Lemma 2.5, imply, for every fixed $D>0$,
\[
\mathbb P\left(\|X\|>2+Cb_n^{-1/4}(\log n)^{3/4}
+Cb_n^{-1/2}\sqrt{D\log n}\right)\leq n^{-D}.
\]
Since $b_n\geq n^\epsilon$, taking $D=12$ shows that $\|X\|\leq2.5$ with probability at least $1-n^{-12}$ for all sufficiently large $n$. The intersection of this event with \eqref{greenfunctionupperbounds} consequently has probability at least $1-2n^{-12}\geq1-n^{-10}$. On this event all eigenvalues of $X$ have modulus at most $2.5$.

For all sufficiently large $n$, we have $v_n<1$, since $\mathbf c<\epsilon/10$. Let $\psi\in\mathbb{C}^n$ be any normalized right eigenvector of $X$ with eigenvalue $\lambda$. Set $u:=\begin{pmatrix}
    0\\\psi
\end{pmatrix}\in\mathbb{C}^{2n}$. Then $\mathcal Y_\lambda  u=0$, so $u$ is a unit vector in $\ker \mathcal Y_\lambda$. Let $P_0$ be the orthogonal projection onto this kernel. The spectral theorem applied to $\mathcal Y_\lambda$ gives, for any $1\leq k\leq 2n$,
\begin{equation}
    \Im(\mathcal G_\lambda(iv_n))_{kk}=\langle e_k,\frac{v_n}{\mathcal Y_\lambda^2+v_n^2}e_k\rangle\geq \frac{1}{v_n}\|P_0e_k\|_2^2.
\end{equation}Since for each $1\leq j\leq n$ we have, using $u\in\ker\mathcal Y_\lambda$,
$$
|\psi_j|=|\langle e_{n+j},u\rangle|\leq\|P_0e_{n+j}\|_2,
$$ we have
$$
\frac{|\psi_j|^2}{v_n}\leq \Im (\mathcal G_\lambda(iv_n))_{n+j,n+j}\leq|(\mathcal G_\lambda(iv_n))_{n+j,n+j}|\leq C.
$$Therefore,
\[
\|\psi\|_\infty\leq C^{1/2}v_n^{1/2}
=C^{1/2}b_n^{-1/10}n^{\mathbf c/2}
\leq b_n^{-1/10}n^{\mathbf c}
\]
for all sufficiently large $n$, because $C^{1/2}\leq n^{\mathbf c/2}$. Since the event is uniform over $|z|\leq3$, this holds for all normalized right eigenvectors simultaneously and completes the proof.
\end{proof}

\section{Circular law for general profiles and small products}\label{section4}
We prove the circular law for the general variance-profile model and for small products by combining the singular-value counting estimates, least singular value bounds and Kolmogorov comparison of squared singular-value measures to compare normalized logarithmic potentials. This strategy follows \cite{jain2021circular}; earlier versions appear in \cite{silverstein1995empirical}, \cite{MR2567175} and \cite{tao2008random}.

To complete the proof of circular law, we shall use the replacement principle by Tao and Vu \cite{ WOS:000281425000010}.

\begin{theorem}\label{replacements}
    Assume that for each $n$, $X$ and $G$ are two ensembles of size $n\times n$ random matrices satisfying
    \begin{enumerate}
        \item The quantity $$\frac{1}{n}\|G\|_{\text{HS}}^2+\frac{1}{n}\|X\|_{\text{HS}}^2$$ is bounded in probability (resp, almost surely), where $\|\cdot\|_{\text{HS}}$ denotes the Hilbert-Schmidt norm of a square matrix.
        \item For a.e. $z\in\mathbb{C}$ with respect to Lebesgue measure, we have
$$
\frac{1}{n}\log|\det(X-zI_n)|  -\frac{1}{n}\log|\det(G-zI_n)|
$$ converges to 0 in probability (resp, almost surely). 
    \end{enumerate} Then $\mu_X-\mu_G$ converges weakly to zero in probability (resp. almost surely).
\end{theorem}

Also recall that for an $n\times n$ matrix $M$, denote by $\sigma_1(M)\geq\sigma_2(M)\cdots\geq\sigma_n(M)$ its singular values, then 
$$
|\det M|=\prod_{i=1}^n\sigma_i(M).
$$
We will also need a useful technical lemma from \cite{jain2021circular}:

\begin{lemma}\label{lemma4.2345}
    Consider two probability measures $\mu,\nu$ on $\mathbb{R}$ and $0<a<b$. Then 
    $$
\left|\int_a^b\log(x)d\mu(x)-\int_a^b\log(x)d\nu(x)\right|\leq 2(|\log b|+|\log a|)\|\mu-\nu\|_{[a,b]},
$$ in which we define 
$$\|\mu-\nu\|_{[a,b]}:=\sup_{x\in[a,b]}|\mu([a,x])-\nu([a,x])|.$$
\end{lemma}

We now prove Theorem \ref{circulargeneral} and establish the small-product result used in Theorem \ref{theorem1.616}. The latter theorem is completed in Section \ref{section6}.

\begin{proof}[\proofname of Theorem \ref{circulargeneral}]
    First consider the Gaussian case and fix $z\in\mathbb C\setminus\{0\}$. Choose $\mathbf c,\kappa>0$ sufficiently small that $\mathbf c+3\kappa<6\epsilon/5$ and $n^{2\kappa}/\sqrt{b_n}\to0$. We truncate the sum
    $$
\frac{1}{n}\sum_{i=1}^n\log\sigma_i(X_z)=\frac{1}{n}\sum_{i=1}^n\log\sigma_i(X_z)1_{\{\sigma_i(X_z)\geq (b_n)^{-1/5}n^\mathbf{c}\}}+\frac{1}{n}\sum_{i=1}^n\log\sigma_i(X_z) 1_{\{\sigma_i(X_z)< (b_n)^{-1/5}n^\mathbf{c}\}}.
    $$Set $$
a_n:=(b_n)^{-1/5}n^{\mathbf c}.$$
    By Corollary \ref{corollary3.45} applied to $I=[-a_n,a_n]$, with probability $1-n^{-10}$ there are at most $C_2 n(b_n)^{-1/5}n^\mathbf{c}$ terms in the second summation.

    To apply Theorem \ref{tikhoroveheorem}, choose $R>\max(1,|z|^{-1})$. Here $\sigma=1$ and $\sigma^*=b_n^{-1/2}$, so its shift condition holds for all sufficiently large $n$. For a general complex Gaussian atom, a fixed phase rotation diagonalizes the covariance of its real and imaginary parts, making them independent, or making the atom real-valued in the degenerate case. Rotating $X$ and $z$ together preserves $|z|$ and all singular values. Theorem \ref{tikhoroveheorem} therefore applies. Together with the norm bound $\sigma_1(X_z)\leq K_\xi+|z|$ from \eqref{generalnormbound}, it gives, with probability $1-o(1)$, $|\log\sigma_i(X_z)|\leq \frac{n^{1+3\kappa}}{b_n}\log b_n$ for every $i\in[n]$ and all sufficiently large $n$; the factor $\log b_n$ absorbs the fixed $z$-dependent constants. The choice $\mathbf c+3\kappa<6\epsilon/5$ and the assumption $b_n\geq n^{5/6+\epsilon}$ now give:
\begin{equation}\label{sumoversigmai}
\left|\frac{1}{n}\sum_{i=1}^n\log\sigma_i(X_z) 1_{\{\sigma_i(X_z)< (b_n)^{-1/5}n^\mathbf{c}\}}\right|\leq C_2\frac{1}{n} n^\mathbf{c}\frac{n}{(b_n)^{1/5}}\frac{n^{1+3\kappa}}{b_n} \log b_n=o(1),\quad n\to\infty.
\end{equation} The same estimate in \eqref{sumoversigmai} is true if we replace $X_z$ by $G_z$ in the sum, with a much simpler proof.

Now we consider the first sum difference
\begin{equation}\label{sumsumsumsumsum}\begin{aligned}&\left|
\frac{1}{n}\sum_{i=1}^n\log\sigma_i(X_z)1_{\{\sigma_i(X_z)\geq (b_n)^{-1/5}n^\mathbf{c}\}}-\frac{1}{n}\sum_{i=1}^n\log\sigma_i(G_z)1_{\{\sigma_i(G_z)\geq (b_n)^{-1/5}n^\mathbf{c}\}}\right|\\&\leq 2C_z(1+\log n)\|\nu_{X_z}-\nu_{G_z}\|_{[0,\infty)}=o(1),\quad n\to\infty,
\end{aligned}\end{equation}Here we applied Lemma \ref{lemma4.2345} to the squared
singular-value measures $\nu_{X_z}$ and $\nu_{G_z}$ on the interval
$[a_n^2,K_z^2]$, where $
a_n=(b_n)^{-1/5}n^{\mathbf c},$
and $K_z$ is a deterministic upper bound for the singular values on
the event under consideration. We also used
$\log\sigma=\frac12\log(\sigma^2)$. Theorem
\ref{theorem5.11} then shows that the right-hand side is $o(1)$. Here one may take $K_z=K_\xi+|z|$ on the common norm event for $X$ and $G$ provided by \eqref{generalnormbound}.
Thus, for every fixed $z\in\mathbb C\setminus\{0\}$, the second criterion of Theorem \ref{replacements} holds in probability. Since $\{0\}$ has planar Lebesgue measure zero, this supplies the required almost-everywhere statement. The doubly stochastic variance normalization gives $\mathbb E[n^{-1}\|X\|_{\mathrm{HS}}^2]=\mathbb E[n^{-1}\|G\|_{\mathrm{HS}}^2]=1$, so Markov's inequality verifies the first criterion. The circular law for the complex Ginibre reference $G$ then yields the circular law for $X$.

For the non-Gaussian case, fix again $z\neq0$ and choose $\mathbf c,\kappa>0$ sufficiently small that $\mathbf c+3\kappa<9\epsilon/8$ and $n^{2\kappa}/\sqrt{b_n}\to0$. The same argument applies with Corollary \ref{corollary3.45news} in place of Corollary \ref{corollary3.45}, giving:
\begin{equation}\label{sumoversigmailow}
\left|\frac{1}{n}\sum_{i=1}^n\log\sigma_i(X_z) 1_{\{\sigma_i(X_z)< (b_n)^{-1/8}n^\mathbf{c}\}}\right|\leq C_2\frac{1}{n} n^\mathbf{c}\frac{n}{(b_n)^{1/8}}\frac{n^{1+3\kappa}}{b_n}\log b_n=o(1),\quad n\to\infty,
\end{equation} where the right-hand side tends to zero by $b_n\geq n^{8/9+\epsilon}$ and the joint choice of $\mathbf c$ and $\kappa$.
\end{proof}

We use the following smoothing estimate to avoid a fixed operator-norm bound in the small-product comparison.

\begin{lemma}[Stieltjes smoothing with a compact reference measure]\label{productstieltjessmoothing}
Let $\nu$ be a probability measure supported in $[-B,B]$, where $B\geq1$, with a density bounded by $M$, and let $\mu$ be any probability measure on $\mathbb R$. For $\rho=\mu,\nu$, write
\[
m_\rho(\eta)=\int_{\mathbb R}\frac{1}{x-\eta}\,d\rho(x),
\qquad \eta\in\mathbb C_+.
\]
Set $A=B+3$. For $0<v\leq1$, define
\[
\Delta=\sup_{|u|\leq A}|m_\mu(u+iv)-m_\nu(u+iv)|.
\]
Then
\[
\|\mu-\nu\|_{(-\infty,\infty)}\leq C_{B,M}(\Delta+\sqrt v).
\]
\end{lemma}
\begin{proof}
Let $P_v$ have density $p_v(t)=v/(\pi(t^2+v^2))$, and write $Q_\rho=\rho*P_v$ for $\rho=\mu,\nu$. Its density is $\pi^{-1}\Im m_\rho(u+iv)$. Since $\nu$ is supported in $[-B,B]$, each tail of $Q_\nu$ outside $[-A,A]$ has mass at most $v/(3\pi)$. Integration of the Stieltjes comparison over $[-A,A]$ gives
\[
Q_\mu(\mathbb R\setminus[-A,A])\leq\frac{2v}{3\pi}+\frac{2A}{\pi}\Delta.
\]
Anchor the two distribution functions at $-A$ and integrate their density difference. On $[-A,A]$ this yields
\[
|F_{Q_\mu}(x)-F_{Q_\nu}(x)|\leq\frac v\pi+\frac{4A}{\pi}\Delta.
\]
The same bound holds outside this interval by the tail estimates.

For $Y$ with law $P_v$ and any $h>0$, coupling gives
\[
F_{Q_\rho}(x-h)-p_h\leq F_\rho(x)\leq F_{Q_\rho}(x+h)+p_h,
\qquad p_h=\mathbb P(Y>h)=\mathbb P(Y<-h)\leq\frac{v}{\pi h}.
\]
Applying these inequalities to both measures and using the density bound for $\nu$, we obtain
\[
\|\mu-\nu\|_{(-\infty,\infty)}
\leq\frac v\pi+\frac{4A}{\pi}\Delta+2Mh+\frac{2v}{\pi h}.
\]
Taking $h=\sqrt v$ proves the claim.
\end{proof}

\begin{Proposition}[Small-product regime]\label{productsmallm}
Let $X_1,\ldots,X_{m_n}$ be independent $n\times n$ matrices with
i.i.d. entries distributed as $\xi/\sqrt n$, and assume
$m_n\leq n^{1/8-\epsilon}$ for some fixed $\epsilon>0$.
Suppose $\mathbb E\xi=0$, $\mathbb E|\xi|^2=1$, and
$\mathbb E|\xi|^3<\infty$.
Assume either that $\xi$ is real with a density bounded by $L$,
or that its real and imaginary parts are independent and at least
one has a density bounded by $L$. Then the root empirical spectral
measure $\nu_{P_n}^{(m_n)}$ defined in Theorem \ref{theorem1.616}
converges in probability to the circular law.
\end{Proposition}

\begin{proof}
Let $\mathcal L$ be the cyclic linearization defined in
\eqref{welinearizeit}. By the standard determinant identity for cyclic
linearizations, we have
\[
\det\bigl(zI_{nm_n}-\mathcal L\bigr)
=
\det\bigl(z^{m_n}I_n-P_n\bigr).
\]
Consequently, the eigenvalues of $\mathcal L$, counted with algebraic
multiplicity, are precisely all the $m_n$-th roots of the eigenvalues
of $P_n$. In particular,
\[
\mu_{\mathcal L}=\nu_{P_n}^{(m_n)}
\]
for every realization of $X_1,\ldots,X_{m_n}$. 

Thus it suffices to prove the circular law for $\mathcal L$. Write $m=m_n$ and $N=nm$, and let $G$ be an $N\times N$ complex Ginibre matrix with entry variance $N^{-1}$. Fix $z\neq0$. We first justify the singular-value comparison under the moment assumptions of this proposition, without invoking the subgaussian hypothesis of Theorem \ref{theorem5.11}.

Every row and column of $\mathcal L$ has total entry variance one, its entries are independent, and its maximum entry variance is $n^{-1}$. Hence the finite-third-moment estimate \eqref{lines495} applies with ambient dimension $N$ and bandwidth $n$. The associated free Stieltjes transform is the same as for $G$, by \eqref{definitiongzeta}--\eqref{agageggwgewe}; denote it by $m^o_{z,\mathrm{free}}$. The corresponding probability measure $\mu_z$ has compact support and bounded density, as in \eqref{boundeddensity}.

For any fixed $A>0$, take $v=n^{-1/32}$. Writing $m_z^{\mathcal L}$ and $m_z^G$ for the normalized Stieltjes transforms of the respective Hermitizations, \eqref{lines495} gives, for each $\eta=u+iv$,
\[
\begin{aligned}
|m_z^{\mathcal L}(\eta)-m^o_{z,\mathrm{free}}(\eta)|
&\leq \frac{C}{\sqrt n\,v^4}+\frac{C}{nv^5}
+\frac{C_D\sqrt{\log N}}{\sqrt N\,v}
+\frac{C\sqrt{\log N}}{\sqrt n\,v^2}\\
&=O(n^{-3/8})
\end{aligned}
\]
with probability at least $1-N^{-D}$. Here $n\leq N\leq n^{9/8}$, and the constants are independent of $m$. The same bound holds for $m_z^G$, whose bandwidth is $N\geq n$. For each deterministic $u$, the constants in \eqref{lines495} are uniform on any fixed interval; the probability bound above is pointwise in $u$. Using an $n^{-1}$-net of $[-A,A]$, the derivative bound $|m'(\eta)|\leq v^{-2}$, and a sufficiently large fixed $D$, we obtain with probability $1-o(1)$
\begin{equation}\label{productstieltjescomparison}
\sup_{|u|\leq A}\left(
|m_z^{\mathcal L}(u+iv)-m^o_{z,\mathrm{free}}(u+iv)|
+|m_z^G(u+iv)-m^o_{z,\mathrm{free}}(u+iv)|
\right)=O(n^{-3/8}).
\end{equation}

Choose $B_z\geq1$ such that $\mu_z$ is supported in $[-B_z,B_z]$, and take $A=B_z+3$ in \eqref{productstieltjescomparison}. Lemma \ref{productstieltjessmoothing} applies separately to each Hermitization law with reference measure $\mu_z$, without imposing any support condition on the empirical measures. Thus each law is within $O(n^{-3/8}+\sqrt v)=O(n^{-1/64})$ of $\mu_z$ in Kolmogorov distance. Passing to squared singular values as in \eqref{convergencereduced} costs at most a fixed factor: the mass of $[-\sqrt t,\sqrt t]$ is the difference of a distribution function and its left limit. We obtain
\begin{equation}\label{productkolmogorovcomparison}
\|\nu_{\mathcal L-zI_N}-\nu_{G-zI_N}\|_{[0,\infty)}
=O(n^{-1/64})
\end{equation}
with probability $1-o(1)$. This supplies the required comparison for the growing block family under the stated moment assumptions.

We next control the logarithm near zero. Choose $0<\delta<\min(\epsilon,1/8)$ and set $a=n^{-1/8+\delta}$. We spell out the small-value count underlying Corollary \ref{corollary3.45news} for the present fixed $z$. Apply \eqref{lines495} at $\eta=ia$, with ambient dimension $N$ and bandwidth $n$. Its first term is $O(n^{-4\delta})$, and the other terms also tend to zero. Since $\mu_z$ has bounded density, $\Im m^o_{z,\mathrm{free}}(ia)=O_z(1)$. Thus, with probability $1-o(1)$, $\Im m_z^{\mathcal L}(ia)=O_z(1)$. The same argument applies to $G$, whose bandwidth is $N\geq n$. By the spectral theorem, for either $M=\mathcal L-zI_N$ or $M=G-zI_N$,
\[
\frac1N\#\{j:\sigma_j(M)\leq a\}
\leq \frac{2a}{N}\sum_{j=1}^N\frac{a}{\sigma_j(M)^2+a^2}
\leq C_za.
\]

Set $q=\min(|z|,|z|^{-1})$. Theorem \ref{theorem2.2whatsapp}, first for the $m$ blocks of size $n$ and then for the single Ginibre block of size $N$, gives with probability $1-o(1)$
\[
s_{min}(\mathcal L-zI_N)\geq q^{2m}n^{-25m},
\qquad s_{min}(G-zI_N)\geq q^2N^{-25}.
\]
Consequently, for either $M=\mathcal L-zI_N$ or $M=G-zI_N$, every singular value below $a<1$ satisfies $|\log\sigma_j(M)|\leq C_zm\log n$. Combining the count and the lower bound gives
\begin{equation}\label{productsmalllog}
\frac1N\sum_{\sigma_j(M)\leq a}|\log\sigma_j(M)|
\leq C_zam\log n
\leq C_zn^{\delta-\epsilon}\log n=o(1).
\end{equation}

We only need a polynomial norm bound to compare the remaining logarithmic integrals. By Lemma \ref{linearizationgoodevent}, with probability $1-O(n^{-2})$,
\[
\|\mathcal L\|=\max_{j\in[m]}\|X_j\|\leq n^2.
\]
Also, $\mathbb E\|G\|_{\mathrm{HS}}^2=N$, so Markov's inequality gives $\mathbb P(\|G\|>N^2)\leq N^{-3}$. Set $R_{n,z}=N^2+|z|+1$. On this common norm event, both squared singular-value measures are supported in $[0,R_{n,z}^2]$. Define $f_a(t)=\frac12\log\max(t,a^2)$ for $t\geq0$. Integration by parts, as in Lemma \ref{lemma4.2345}, and \eqref{productkolmogorovcomparison} imply
\[
\begin{aligned}
\left|\int f_a\,d\nu_{\mathcal L-zI_N}-\int f_a\,d\nu_{G-zI_N}\right|
&\leq \log(R_{n,z}/a)\,
\|\nu_{\mathcal L-zI_N}-\nu_{G-zI_N}\|_{[0,\infty)}\\
&=O(n^{-1/64}\log n)=o(1).
\end{aligned}
\]
Here $N\leq n^{9/8}$, so $\log(R_{n,z}/a)=O_z(\log n)$. For either matrix $M$, replacing $N^{-1}\log|\det M|$ by $\int f_a\,d\nu_M$ costs at most the left-hand side of \eqref{productsmalllog}. The polynomial norm event controls all large singular values. We have therefore proved, for every fixed $z\neq0$,
\[
\frac1N\log|\det(\mathcal L-zI_N)|
-\frac1N\log|\det(G-zI_N)|\longrightarrow0
\quad\text{in probability}.
\]
Finally, $\mathbb E[N^{-1}\|\mathcal L\|_{\mathrm{HS}}^2]=\mathbb E[N^{-1}\|G\|_{\mathrm{HS}}^2]=1$, so the norm tightness condition of Theorem \ref{replacements} follows from Markov's inequality. The exceptional point $z=0$ has Lebesgue measure zero. The replacement principle and the circular law for $G$ now imply the circular law for $\mathcal L$, and hence the claimed convergence of $\nu_{P_n}^{(m_n)}$.
\end{proof}

\section{Least singular value of the cyclic linearization}\label{section5}
In this section, we prove Theorem \ref{theorem2.2whatsapp}. Write $m=m_n$, $N=nm$, $A=\mathcal L-zI_N$ and $q=\min(|z|,|z|^{-1})$. Throughout the section, $1\leq m\leq n$ and $0<|z|\leq K'$. The probability estimates are for each deterministic $z$ in this range, with constants uniform in $z$. We use cyclic indices, so that $X_{m+1}=X_1$ and $u_{[m+1]}=u_{[1]}$. In particular, when $m=1$ the linearization is $\mathcal L=X_1$.

Set
\[
\mathcal E_n=\{\|X_j\|\leq n^2,\ s_{min}(X_j)\geq n^{-5}\text{ for every }j\in[m]\}.
\]
We first record the probability estimate for this event, including its dependence on the number of blocks.

\begin{lemma}\label{linearizationgoodevent}
There is a constant $C_L>0$, depending only on $L$, such that
\[
\mathbb P(\mathcal E_n^c)\leq C_Lmn^{-3}\leq C_Ln^{-2}.
\]
\end{lemma}
\begin{proof}
The normalization $\mathbb E|\xi|^2=1$ gives $\mathbb E\|X_j\|_{\mathrm{HS}}^2=n$. Thus Markov's inequality yields
\begin{equation}\label{singleblocknorm}
\mathbb P(\|X_j\|>n^2)\leq\frac{\mathbb E\|X_j\|_{\mathrm{HS}}^2}{n^4}=n^{-3}.
\end{equation}
The deterministic distance reduction in Lemma \ref{rowdistancereduction} and the scalar small-ball estimate in Lemma \ref{complexscalarsmallball}, applied to a single matrix $X_j$, give
\[
\mathbb P(s_{min}(X_j)\leq n^{-5})
\leq n\,C_L\sqrt n\,(\sqrt n\,n^{-5})
=C_Ln^{-3}.
\]
Indeed, for each deleted row, condition on the other rows and choose a unit vector in their right kernel; its coefficient vector has norm one. A union bound over the $m\leq n$ blocks proves the claim, after increasing $C_L$.
\end{proof}

The target threshold in Theorem \ref{theorem2.2whatsapp} allows us to use a union bound over all rows. No decomposition of the unit sphere is needed.

\begin{lemma}[Reduction to row distances]\label{rowdistancereduction}
Let $B$ be a complex $d\times d$ matrix with rows $R_1,\ldots,R_d$. Set
\[
H_i=\operatorname{span}\{R_j^*:j\neq i\}\subset\mathbb C^d,
\qquad D_i=\operatorname{dist}(R_i^*,H_i).
\]
For any $s\geq0$, if $s_{min}(B)\leq s$, then $D_i\leq\sqrt d\,s$ for some $i\in[d]$.
\end{lemma}
\begin{proof}
Choose a unit vector $v\in\mathbb C^d$ with $\|B^*v\|\leq s$, and an index $i$ with $|v_i|\geq d^{-1/2}$. Projecting $B^*v=\sum_jv_jR_j^*$ onto $H_i^\perp$ gives $|v_i|D_i\leq s$.
\end{proof}

For a vector $u\in\mathbb C^N$, write $u=(u_{[1]},\ldots,u_{[m]})$ with $u_{[j]}\in\mathbb C^n$. Let $A^{(i)}$ denote $A$ with its $i$-th row deleted.

\begin{Proposition}\label{proposition5.32}
On $\mathcal E_n$, for every $i\in[N]$ and every unit vector $u\in\ker A^{(i)}$, we have
\[
\min_{j\in[m]}\|u_{[j]}\|\geq H,
\qquad H:=q^mn^{-10m}m^{-1/2}.
\]
\end{Proposition}
\begin{proof}
Suppose that the deleted row belongs to block row $a$. All the other block rows are intact, and therefore
\[
X_{j+1}u_{[j+1]}=zu_{[j]},\qquad j\neq a.
\]
On $\mathcal E_n$, each such relation gives
\[
\|u_{[j+1]}\|\geq\frac{|z|}{n^2}\|u_{[j]}\|,
\qquad
\|u_{[j]}\|\geq\frac{n^{-5}}{|z|}\|u_{[j+1]}\|.
\]
Both ratios are at least $c:=qn^{-10}\leq1$, since $|z|/n^2\geq qn^{-10}$ and $n^{-5}/|z|\geq qn^{-10}$. Removing the relation indexed by $a$ from the cycle leaves a path connecting all $m$ blocks. Some block has norm at least $m^{-1/2}$, and every other block is reached along at most $m-1$ edges. Thus
\[
\|u_{[j]}\|\geq c^{m-1}m^{-1/2}\geq H
\]
for every $j$. When $m=1$, the sole block has norm one and the same conclusion holds.
\end{proof}

\begin{lemma}[Scalar small-ball estimate]\label{complexscalarsmallball}
Let $\xi_1,\ldots,\xi_d$ be independent copies of an atom satisfying the density assumption of Theorem \ref{theorem2.2whatsapp}. There is a constant $C_L$, depending only on $L$, such that for every nonzero $a\in\mathbb C^d$, every $w\in\mathbb C$ and every $t\geq0$,
\[
\mathbb P\left(\left|\sum_{k=1}^da_k\xi_k-w\right|\leq t\right)
\leq\frac{C_Lt}{\|a\|}.
\]
\end{lemma}
\begin{proof}
For real atoms, one of the real coefficient vectors $\Re a$ and $\Im a$ has norm at least $\|a\|/\sqrt2$. Project onto the corresponding real component and apply the one-dimensional case of \cite{rudelson2015small}, Theorem 1.1. For complex atoms, choose a component with bounded density and condition on all copies of the other component. Then apply the same real projection argument. The independence of the two components permits this conditioning.
\end{proof}

\begin{proof}[\proofname\ of Theorem \ref{theorem2.2whatsapp}]
Let $R_i$ be the $i$-th row of $A$, and define $H_i,D_i$ as in Lemma \ref{rowdistancereduction}. Condition on the sigma-algebra $\mathscr F_i$ generated by all rows other than row $i$. Choose a unit vector $u_i\in\ker A^{(i)}$ measurably with respect to $\mathscr F_i$; for instance, normalize the projection of the first standard basis vector having nonzero projection onto this kernel. Then $u_i\perp H_i$ and
\[
D_i\geq|R_iu_i|.
\]
If row $i$ belongs to block row $a$, its random entries lie in block column $b=a+1$, with cyclic indices. Set $a_i=(u_i)_{[b]}$. Given $\mathscr F_i$, the scalar $R_iu_i$ has the form
\[
\frac1{\sqrt n}\sum_{k=1}^n(a_i)_k\xi_k-z(u_i)_i,
\]
where the $\xi_k$ remain independent copies of $\xi$. By Proposition \ref{proposition5.32}, on $\mathcal E_n$ we have $\|a_i\|\geq H$. The event that this coefficient bound holds is $\mathscr F_i$-measurable. Consequently, for any $d\geq0$, Lemma \ref{complexscalarsmallball} gives
\begin{equation}\label{linearizationrowbound}
\begin{aligned}
\mathbb P(\mathcal E_n\cap\{D_i\leq d\})
&\leq\mathbb E\left[\mathbf 1_{\{\|a_i\|\geq H\}}
\mathbb P(|R_iu_i|\leq d\mid\mathscr F_i)\right]\\
&\leq C_L\sqrt n\,d/H.
\end{aligned}
\end{equation}
Here we have replaced $\mathcal E_n$ by the measurable coefficient event before conditioning, so no independence conditional on $\mathcal E_n$ is asserted.

Take $s=q^{2m}n^{-25m}$. Lemma \ref{rowdistancereduction}, \eqref{linearizationrowbound} and a union bound over the $N$ rows yield
\[
\begin{aligned}
\mathbb P(\mathcal E_n\cap\{s_{min}(A)\leq s\})
&\leq C_LN\sqrt n\,\sqrt N\,s/H\\
&=C_Ln^2m^2q^mn^{-15m}
\leq C_Ln^{-11},
\end{aligned}
\]
where we used $m\leq n$, $m\geq1$ and $q\leq1$. Adding $\mathbb P(\mathcal E_n^c)\leq C_Ln^{-2}$ gives an $O_L(n^{-2})$ bound, which implies the stated estimate. All constants are uniform for deterministic $z$ with $0<|z|\leq K'$ and depend only on $L$.
\end{proof}

\section{Product circular law without a restriction on the number of factors}\label{section6}

We work under the assumptions of Theorem \ref{theorem1.616} and complete its proof by combining Proposition \ref{productsmallm} with the singular-value law of Hanin and Jiang \cite[Theorem 1]{hanin2025global}. Their main result states that, for a product $Q_{k,n}$ of $k$ independent $n\times n$ matrices with i.i.d. entries distributed as $\xi/\sqrt n$,
\[
\frac1n\sum_{j=1}^n\delta_{\sigma_j(Q_{k,n})^{2/k}}
\longrightarrow U[0,1]\quad\text{in probability}
\]
as $k,n\to\infty$, with no restriction on their relative growth. Thus the squared singular values, after taking their $k$-th roots, have a universal uniform limit. Their theorem gives quantitative bounds in Kolmogorov distance under a finite fourth moment and a bounded-density assumption. In the complex case, the extension in Section 4 of their paper assumes independent real and imaginary parts, both with bounded densities, as in Theorem \ref{theorem1.616}.

This result is useful here because, after conditioning on all but the last factor, we can compare the shifted log determinant of $P_n$ with a clipped logarithmic statistic of the singular values of the remaining product. The comparison error, after division by $nm_n$, is $O(\log n/m_n)$. The Hanin--Jiang law identifies the limiting statistic as the logarithmic potential of the circular law; the cyclic linearization and the replacement principle then give the desired spectral limit. In fact, this argument applies whenever $m_n/\log n\to\infty$. The cutoff $m_n>n^{1/16}$ below is only a convenient choice for joining it with Proposition \ref{productsmallm}.

\begin{proof}[\proofname\ of Theorem \ref{theorem1.616}]
Proposition \ref{productsmallm}, with $\epsilon=1/16$, treats the indices
for which $m=m_n\leq n^{1/16}$. We prove the assertion along the
remaining indices, so throughout this proof $m>n^{1/16}$. This convenient split joins the two proofs and ensures $\log n/m\to0$, which controls the determinant comparison error and the rooted singular-value support below. Write
$P_n=AX$, where $A=X_1\cdots X_{m-1}$ and $X=X_m$.

We first compare the shifted determinant with a clipped singular-value
sum. For each deterministic $D\in\mathbb C^{n\times n}$,
Lemmas \ref{rowdistancereduction} and \ref{complexscalarsmallball} give
\begin{equation}\label{arbitraryshiftleast}
\mathbb P\bigl(s_{min}(X-D)\leq t\bigr)\leq C_Ln^2t,
\qquad t\geq0.
\end{equation}
The constant $C_L$ is independent of $D$. Indeed, after conditioning on all rows except row $i$, choose a unit
vector in the kernel of the other rows. The scalar small-ball bound
for its inner product with row $i$ is $C_L\sqrt n\,d$ at distance
$d$, irrespective of the deterministic shift. The row-distance
reduction takes $d=\sqrt n\,t$, and a union bound over $n$ rows
proves \eqref{arbitraryshiftleast}. In particular, each factor and
hence $A$ are almost surely invertible.

Fix $z\neq0$ and set $w=z^m$ and $D=wA^{-1}$. Conditional on $A$, \eqref{arbitraryshiftleast} applies to this $D$. Since $X$ is independent of $A$ and $\mathbb E\|X\|_{\mathrm{HS}}^2=n$, the following event has conditional probability at least $1-C_L'n^{-3}$ given $A$:
\[
\|X\|\leq n^2,\qquad s_{min}(X-D)\geq n^{-5}.
\]
On this event, write $d_j=\sigma_j(D)$, in decreasing order.
Weyl's inequality gives
$\sigma_j(X-D)\leq d_j+n^2$ and, when $d_j>2n^2$,
$\sigma_j(X-D)\geq d_j/2$. For the remaining indices use the
least singular value bound. Thus, for $n\geq2$,
\[
n^{-8}\max(1,d_j)\leq \sigma_j(X-D)
\leq n^3\max(1,d_j).
\]
Using $\det(P_n-wI_n)=\det A\,\det(X-D)$ and the reciprocal
singular values of $A^{-1}$, we obtain
\begin{equation}\label{productdetbridge}
\left|\frac1{nm}\log|\det(P_n-z^mI_n)|
-\frac1{nm}\sum_{j=1}^n\log\max\bigl(\sigma_j(A),|z|^m\bigr)\right|
\leq\frac{8\log n}{m}=o(1).
\end{equation}
For each deterministic $w$, the conditional failure probability is at most $C_L'n^{-3}$, with a constant independent of $A$, $m$ and $w$. This is a pointwise bound in $w$; no union over the factors is used.

Set $\ell=m-1$. By Theorem 1 of \cite{hanin2025global} and its
complex extension in Section 4 of the same paper,
\[
\rho_{\ell,n}:=\frac1n\sum_{j=1}^n
\delta_{\sigma_j(A)^{2/\ell}}
\ \longrightarrow\ U[0,1]
\quad\text{in probability},
\]
since $\ell,n\to\infty$, with no restriction on their relative growth.
We next justify the use of a logarithmic test function. Independence and
$\mathbb E[X_jX_j^*]=I_n$ give
$\mathbb E\|A\|_{\mathrm{HS}}^2=n$, for every $\ell$.
Therefore $\mathbb P(\|A\|>n^2)\leq n^{-3}$, and outside this event
the support of $\rho_{\ell,n}$ is contained in
\[
[0,\exp(4\log n/\ell)]\subset[0,2]
\]
for all sufficiently large $n$.
For fixed $r=|z|>0$, the functions
$\frac{\ell}{2m}\log\max(t,r^{2m/\ell})$ converge uniformly on
$[0,2]$ to $\frac12\log\max(t,r^2)$. They are continuous at zero
because the clipping threshold is positive. Consequently,
\begin{equation}\label{productclippedpotential}
\begin{aligned}
\frac1{nm}\sum_{j=1}^n\log\max\bigl(\sigma_j(A),r^m\bigr)
&=\frac{\ell}{2m}\int\log\max(t,r^{2m/\ell})\,d\rho_{\ell,n}(t)\\
&\longrightarrow\frac12\int_0^1\log\max(t,r^2)\,dt
=:\Phi(r)
=\begin{cases}
(r^2-1)/2,&0<r\leq1,\\
\log r,&r\geq1,
\end{cases}
\end{aligned}
\end{equation}
in probability. The function $z\mapsto\Phi(|z|)$ is the logarithmic potential of the circular law.

Let $N=nm$ and let $\mathcal L$ be the cyclic linearization
\eqref{welinearizeit}. Its determinant identity from the proof of
Proposition \ref{productsmallm}, together with
\eqref{productdetbridge}--\eqref{productclippedpotential}, shows that
\[
\frac1N\log|\det(\mathcal L-zI_N)|
\longrightarrow\Phi(|z|)
\quad\text{in probability}
\]
for every fixed $z\neq0$. An $N\times N$ complex Ginibre matrix $G_N$
with entry variance $N^{-1}$ has the same shifted log-determinant
limit, by \cite{tao2015local}, Theorem 33.
Moreover,
\[
\mathbb E\!\left[\frac1N\|\mathcal L\|_{\mathrm{HS}}^2\right]
=\mathbb E\!\left[\frac1N\|G_N\|_{\mathrm{HS}}^2\right]=1.
\]
Markov's inequality therefore gives the required Hilbert--Schmidt tightness, and the excluded point $z=0$ has planar Lebesgue measure zero. Thus the replacement principle, Theorem \ref{replacements}, applies along the dimensions $N=nm\to\infty$ and proves the circular law for
$\mathcal L$. Since $\mu_{\mathcal L}=\nu_{P_n}^{(m_n)}$, this
proves the required assertion on the large-$m$ indices. Together
with the small-$m$ indices treated at the beginning, it proves the
theorem for every sequence $(m_n)$.
\end{proof}

\printbibliography

@article{ WOS:000281425000010,
Author = {Tao, Terence and Vu, Van and Krishnapur, Manjunath},
Title = {Random matrices: universality of ESDs and the circular law},
Journal = {Annals of Probability},
Year = {2010},
Volume = {38},
Number = {5},
Pages = {2023-2065},
}

@article{haagerup2005new,
author = {Haagerup, Uffe and Thorbjørnsen, Steen},
year = {2003},
month = {01},
pages = {},
title = {A new application of Random Matrices: $Ext(C*_{red}(F_2))$ is not a group},
volume = {162},
journal = {Annals of Mathematics}
}

@article{tikhomirov2023pseudospectrum,
  title={On pseudospectrum of inhomogeneous non-Hermitian random matrices},
  author={Tikhomirov, Konstantin},
  journal={arXiv preprint arXiv:2307.08211},
  year={2023}
}

@article{jain2021circular,
  title={Circular law for random block band matrices with genuinely sublinear bandwidth},
  author={Jain, Vishesh and Jana, Indrajit and Luh, Kyle and O’Rourke, Sean},
  journal={Journal of Mathematical Physics},
  volume={62},
  number={8},
  year={2021},
  publisher={AIP Publishing}
}

@article {MR4303893,
    AUTHOR = {Chaudhuri, Rohit and Jain, Vishesh and Pillai, Natesh S.},
     TITLE = {Universality and least singular values of random matrix
              products: a simplified approach},
   JOURNAL = {Bernoulli},
  FJOURNAL = {Bernoulli. Official Journal of the Bernoulli Society for
              Mathematical Statistics and Probability},
    VOLUME = {27},
      YEAR = {2021},
    NUMBER = {4},
     PAGES = {2519--2531},
      ISSN = {1350-7265,1573-9759},
   MRCLASS = {60B20 (15B52)},
  MRNUMBER = {4303893},
MRREVIEWER = {Florent\ Benaych-Georges},
}

@article{bandeira2023matrix,
  title={Matrix concentration inequalities and free probability},
  author={Bandeira, Afonso S and Boedihardjo, March T and van Handel, Ramon},
  journal={Inventiones mathematicae},
  volume={234},
  number={1},
  pages={419--487},
  year={2023},
  publisher={Springer}
}

@article{girkoarticle,
author = {Girko, V.},
year = {2004},
month = {09},
pages = {49-104},
title = {The Strong Circular Law. Twenty years later. Part I},
volume = {12},
journal = {Random Operators and Stochastic Equations}}

@article{helton2007operator,
  title={Operator-valued semicircular elements: solving a quadratic matrix equation with positivity constraints},
  author={Helton, J William and Far, Reza Rashidi and Speicher, Roland},
  journal={International Mathematics Research Notices},
  volume={2007},
  number={9},
  pages={rnm086--rnm086},
  year={2007},
  publisher={OUP}
}

@article{tao2008random,
  title={Random matrices: the circular law},
  author={Tao, Terence and Vu, Van},
  journal={Communications in Contemporary Mathematics},
  volume={10},
  number={02},
  pages={261--307},
  year={2008},
  publisher={World Scientific}
}

@book {MR2567175,
    AUTHOR = {Bai, Zhidong and Silverstein, Jack W.},
     TITLE = {Spectral analysis of large dimensional random matrices},
    SERIES = {Springer Series in Statistics},
   EDITION = {Second},
 PUBLISHER = {Springer, New York},
      YEAR = {2010},
     PAGES = {xvi+551},
      ISBN = {978-1-4419-0660-1},
   MRCLASS = {60B20 (15B52 62H99 91G70 94A05)},
  MRNUMBER = {2567175},
MRREVIEWER = {Wenbo\ V.\ Li},
}

@article {MR3945840,
    AUTHOR = {Rudelson, Mark and Tikhomirov, Konstantin},
     TITLE = {The sparse circular law under minimal assumptions},
   JOURNAL = {Geom. Funct. Anal.},
  FJOURNAL = {Geometric and Functional Analysis},
    VOLUME = {29},
      YEAR = {2019},
    NUMBER = {2},
     PAGES = {561--637},
      ISSN = {1016-443X,1420-8970},
   MRCLASS = {60B20 (15A18)},
  MRNUMBER = {3945840},
MRREVIEWER = {Vladislav\ Kargin},
}

@article {MR1428519,
    AUTHOR = {Bai, Z. D.},
     TITLE = {Circular law},
   JOURNAL = {Ann. Probab.},
  FJOURNAL = {The Annals of Probability},
    VOLUME = {25},
      YEAR = {1997},
    NUMBER = {1},
     PAGES = {494--529},
      ISSN = {0091-1798,2168-894X},
   MRCLASS = {60F15 (62H99)},
  MRNUMBER = {1428519},
}

@article{rudelson2015small,
  title={Small ball probabilities for linear images of high-dimensional distributions},
  author={Rudelson, Mark and Vershynin, Roman},
  journal={International Mathematics Research Notices},
  volume={2015},
  number={19},
  pages={9594--9617},
  year={2015},
  publisher={Oxford University Press}
}

@article {MR2861673,
    AUTHOR = {O'Rourke, Sean and Soshnikov, Alexander},
     TITLE = {Products of independent non-{H}ermitian random matrices},
   JOURNAL = {Electron. J. Probab.},
  FJOURNAL = {Electronic Journal of Probability},
    VOLUME = {16},
      YEAR = {2011},
     PAGES = {no. 81, 2219--2245},
      ISSN = {1083-6489},
   MRCLASS = {60B20},
  MRNUMBER = {2861673},
MRREVIEWER = {Mattia\ Cafasso},
}

@article {MR2908617,
    AUTHOR = {Bordenave, Charles and Chafa\"i, Djalil},
     TITLE = {Around the circular law},
   JOURNAL = {Probab. Surv.},
  FJOURNAL = {Probability Surveys},
    VOLUME = {9},
      YEAR = {2012},
     PAGES = {1--89},
      ISSN = {1549-5787},
   MRCLASS = {60B20 (15B52 60F15)},
  MRNUMBER = {2908617},
MRREVIEWER = {Vladislav\ Kargin},
}

@article {MR3405592,
    AUTHOR = {Rudelson, Mark and Vershynin, Roman},
     TITLE = {Delocalization of eigenvectors of random matrices with
              independent entries},
   JOURNAL = {Duke Math. J.},
  FJOURNAL = {Duke Mathematical Journal},
    VOLUME = {164},
      YEAR = {2015},
    NUMBER = {13},
     PAGES = {2507--2538},
      ISSN = {0012-7094,1547-7398},
   MRCLASS = {60B20 (15B52)},
  MRNUMBER = {3405592},
MRREVIEWER = {Julio\ Andrade},
}

@article{han2024outliers,
  title={Outliers and bounded rank perturbation for non-Hermitian random band matrices},
  author={Han, Yi},
  journal={arXiv preprint arXiv:2408.00567},
  year={2024}
}

@article{silverstein1995empirical,
  title={On the empirical distribution of eigenvalues of a class of large dimensional random matrices},
  author={Silverstein, Jack W and Bai, Zhi Dong},
  journal={Journal of Multivariate analysis},
  volume={54},
  number={2},
  pages={175--192},
  year={1995},
  publisher={Elsevier}
}

@incollection {MR4680362,
    AUTHOR = {Tikhomirov, Konstantin},
     TITLE = {Quantitative invertibility of non-{H}ermitian random matrices},
 BOOKTITLE = {I{CM}---{I}nternational {C}ongress of {M}athematicians. {V}ol.
              4. {S}ections 5--8},
     PAGES = {3292--3313},
 PUBLISHER = {EMS Press, Berlin},
      YEAR = {2023},
     
}

@article {MR3878135,
    AUTHOR = {Cook, Nicholas and Hachem, Walid and Najim, Jamal and Renfrew,
              David},
     TITLE = {Non-{H}ermitian random matrices with a variance profile ({I}):
              deterministic equivalents and limiting {ESD}s},
   JOURNAL = {Electron. J. Probab.},
  FJOURNAL = {Electronic Journal of Probability},
    VOLUME = {23},
      YEAR = {2018},
     PAGES = {Paper No. 110, 61},
      ISSN = {1083-6489},
   MRCLASS = {60B20 (15A18 15B52)},
  MRNUMBER = {3878135},
MRREVIEWER = {Sandrine\ Dallaporta},
}

@article {MR4195739,
    AUTHOR = {Litvak, Alexander E. and Lytova, Anna and Tikhomirov,
              Konstantin and Tomczak-Jaegermann, Nicole and Youssef, Pierre},
     TITLE = {Circular law for sparse random regular digraphs},
   JOURNAL = {J. Eur. Math. Soc. (JEMS)},
  FJOURNAL = {Journal of the European Mathematical Society (JEMS)},
    VOLUME = {23},
      YEAR = {2021},
    NUMBER = {2},
     PAGES = {467--501},
      ISSN = {1435-9855,1435-9863},
   MRCLASS = {60B20 (05C20 05C50 15B52 47B80)},
  MRNUMBER = {4195739}
}

@article {MR3980923,
    AUTHOR = {Basak, Anirban and Rudelson, Mark},
     TITLE = {The circular law for sparse non-{H}ermitian matrices},
   JOURNAL = {Ann. Probab.},
  FJOURNAL = {The Annals of Probability},
    VOLUME = {47},
      YEAR = {2019},
    NUMBER = {4},
     PAGES = {2359--2416},
      ISSN = {0091-1798,2168-894X},
   MRCLASS = {60B20 (15B52 60B10)},
  MRNUMBER = {3980923},
MRREVIEWER = {Khanh\ Duy\ Trinh},
}

@article{sah2023limiting,
  title={The limiting spectral law for sparse iid matrices},
  author={Sah, Ashwin and Sahasrabudhe, Julian and Sawhney, Mehtaab},
  journal={arXiv preprint arXiv:2310.17635},
  year={2023}
}

@article {MR4388923,
    AUTHOR = {Alt, Johannes and Kr\"uger, Torben},
     TITLE = {Local elliptic law},
   JOURNAL = {Bernoulli},
  FJOURNAL = {Bernoulli. Official Journal of the Bernoulli Society for
              Mathematical Statistics and Probability},
    VOLUME = {28},
      YEAR = {2022},
    NUMBER = {2},
     PAGES = {886--909},
      ISSN = {1350-7265,1573-9759},
   MRCLASS = {60B20 (15B52)},
  MRNUMBER = {4388923},
MRREVIEWER = {Stepan\ Mazur},
}

@article {MR4268303,
    AUTHOR = {Hanin, Boris and Paouris, Grigoris},
     TITLE = {Non-asymptotic results for singular values of {G}aussian
              matrix products},
   JOURNAL = {Geom. Funct. Anal.},
  FJOURNAL = {Geometric and Functional Analysis},
    VOLUME = {31},
      YEAR = {2021},
    NUMBER = {2},
     PAGES = {268--324},
      ISSN = {1016-443X,1420-8970},
   MRCLASS = {60B20 (15B52)},
  MRNUMBER = {4268303},
}

@article {MR3857860,
    AUTHOR = {Cook, Nicholas},
     TITLE = {Lower bounds for the smallest singular value of structured
              random matrices},
   JOURNAL = {Ann. Probab.},
  FJOURNAL = {The Annals of Probability},
    VOLUME = {46},
      YEAR = {2018},
    NUMBER = {6},
     PAGES = {3442--3500},
      ISSN = {0091-1798,2168-894X},
   MRCLASS = {60B20 (15B52)},
  MRNUMBER = {3857860},
MRREVIEWER = {Oleksiy\ Khorunzhiy},
}

@article {MR4580535,
    AUTHOR = {Liu, Dang-Zheng and Wang, Dong and Wang, Yanhui},
     TITLE = {Lyapunov exponent, universality and phase transition for
              products of random matrices},
   JOURNAL = {Comm. Math. Phys.},
  FJOURNAL = {Communications in Mathematical Physics},
    VOLUME = {399},
      YEAR = {2023},
    NUMBER = {3},
     PAGES = {1811--1855},
      ISSN = {0010-3616,1432-0916},
   MRCLASS = {60B20 (37D25)},
  MRNUMBER = {4580535},
}

@article {MR4401507,
    AUTHOR = {Gorin, Vadim and Sun, Yi},
     TITLE = {Gaussian fluctuations for products of random matrices},
   JOURNAL = {Amer. J. Math.},
  FJOURNAL = {American Journal of Mathematics},
    VOLUME = {144},
      YEAR = {2022},
    NUMBER = {2},
     PAGES = {287--393},
      ISSN = {0002-9327,1080-6377},
   MRCLASS = {60B20 (15B52 33E20)},
  MRNUMBER = {4401507},
}

@article {MR4421171,
    AUTHOR = {Ahn, Andrew},
     TITLE = {Fluctuations of {$\beta$}-{J}acobi product processes},
   JOURNAL = {Probab. Theory Related Fields},
  FJOURNAL = {Probability Theory and Related Fields},
    VOLUME = {183},
      YEAR = {2022},
    NUMBER = {1-2},
     PAGES = {57--123},
      ISSN = {0178-8051,1432-2064},
   MRCLASS = {60B20 (15B52 33D52)},
  MRNUMBER = {4421171},
}

@article {MR3770875,
    AUTHOR = {Alt, Johannes and Erd\H os, L\'aszl\'o{} and Kr\"uger, Torben},
     TITLE = {Local inhomogeneous circular law},
   JOURNAL = {Ann. Appl. Probab.},
  FJOURNAL = {The Annals of Applied Probability},
    VOLUME = {28},
      YEAR = {2018},
    NUMBER = {1},
     PAGES = {148--203},
      ISSN = {1050-5164,2168-8737},
   MRCLASS = {60B20 (15B52)},
  MRNUMBER = {3770875},
MRREVIEWER = {Vladislav\ Kargin},
}

@book{girko2012theory,
  title={Theory of stochastic canonical equations: Volumes i and ii},
  author={Girko, Viacheslav Leonidovich},
  volume={535},
  year={2012},
  publisher={Springer Science \& Business Media}
}

@article {MR2663633,
    AUTHOR = {G\"otze, Friedrich and Tikhomirov, Alexander},
     TITLE = {The circular law for random matrices},
   JOURNAL = {Ann. Probab.},
  FJOURNAL = {The Annals of Probability},
    VOLUME = {38},
      YEAR = {2010},
    NUMBER = {4},
     PAGES = {1444--1491},
      ISSN = {0091-1798,2168-894X},
   MRCLASS = {60B20},
  MRNUMBER = {2663633},
MRREVIEWER = {Sasha\ Sodin},
}

@article{brailovskaya2022universality,
  title={Universality and sharp matrix concentration inequalities},
  author={Brailovskaya, Tatiana and van Handel, Ramon},
  journal={Geometric and Functional Analysis},
  pages={1--105},
  year={2024},
  publisher={Springer}
}

@article{wigner1958distribution,
  title={On the distribution of the roots of certain symmetric matrices},
  author={Wigner, Eugene P},
  journal={Annals of Mathematics},
  volume={67},
  number={2},
  pages={325--327},
  year={1958},
  publisher={JSTOR}
}

@article {MR3622892,
    AUTHOR = {Nemish, Yuriy},
     TITLE = {Local law for the product of independent non-{H}ermitian
              random matrices with independent entries},
   JOURNAL = {Electron. J. Probab.},
  FJOURNAL = {Electronic Journal of Probability},
    VOLUME = {22},
      YEAR = {2017},
     PAGES = {Paper No. 22, 35},
      ISSN = {1083-6489},
   MRCLASS = {60B20},
  MRNUMBER = {3622892},
MRREVIEWER = {Khanh\ Duy\ Trinh},
}

@article {MR4112718,
    AUTHOR = {Kopel, Phil and O'Rourke, Sean and Vu, Van},
     TITLE = {Random matrix products: universality and least singular
              values},
   JOURNAL = {Ann. Probab.},
  FJOURNAL = {The Annals of Probability},
    VOLUME = {48},
      YEAR = {2020},
    NUMBER = {3},
     PAGES = {1372--1410},
      ISSN = {0091-1798,2168-894X},
   MRCLASS = {60B20 (15A18 15B52)},
  MRNUMBER = {4112718},
MRREVIEWER = {Ramon\ van Handel},
}

@article{han2025circular1,
  title={The circular law for non-Hermitian random band matrices up to bandwidth $N^{2/3+c}$},
  author={Han, Yi},
  journal={arXiv preprint arXiv:2508.18143v3},
  year={2026}
}

@article{han2026optimal,
  title={The circular law for non-Hermitian random band matrices: optimal bandwidth, periodic profile and discrete law},
  author={Han, Yi},
  journal={arXiv preprint arXiv:2609.01295v2},
  year={2026}
}

@article{han2025circular2,
  title={Circular law for non-Hermitian block band matrices with slowly growing bandwidth},
  author={Han, Yi},
  journal={arXiv preprint arXiv:2511.01744},
  year={2025}
}

@article{yau2025delocalization,
  title={Delocalization of One-Dimensional Random Band Matrices},
  author={Yau, Horng-Tzer and Yin, Jun},
  journal={arXiv preprint arXiv:2501.01718},
  year={2025}
}

@article{shcherbina2025characteristic,
  title={Characteristic polynomials of non-Hermitian random band matrices},
  author={Shcherbina, Mariya and Shcherbina, Tatyana},
  journal={arXiv preprint arXiv:2510.04255},
  year={2025}
}

@article{drogin2025localization,
  title={Localization of One-Dimensional Random Band Matrices},
  author={Drogin, Reuben},
  journal={arXiv preprint arXiv:2508.05802},
  year={2025}
}

@article{dubova2025delocalization,
  title={Delocalization of two-dimensional random band matrices},
  author={Dubova, Sofiia and Yang, Kevin and Yau, Horng-Tzer and Yin, Jun},
  journal={arXiv preprint arXiv:2503.07606},
  year={2025}}

@article{dubova2025delocalization3d,
  title={Delocalization of Non-Mean-Field Random Matrices in Dimensions $ d\geq3$},
  author={Dubova, Sofiia and Yang, Fan and Yau, Horng-Tzer and Yin, Jun},
  journal={arXiv preprint arXiv:2507.20274},
  year={2025}
}

@article{erdHos2025zigzag,
  title={The Zigzag Strategy for Random Band Matrices},
  author={Erd{\H{o}}s, L{\'a}szl{\'o} and Riabov, Volodymyr},
  journal={arXiv preprint arXiv:2506.06441},
  year={2025}
}

@article{bandeira2016sharp,
  author = {Bandeira, Afonso S. and van Handel, Ramon},
  title = {Sharp nonasymptotic bounds on the norm of random matrices with independent entries},
  journal = {The Annals of Probability},
  volume = {44},
  number = {4},
  pages = {2479--2506},
  year = {2016},
  doi = {10.1214/15-AOP1025},
  eprint = {1408.6185},
  archivePrefix = {arXiv}
}

@misc{hanin2025global,
  author = {Hanin, Boris and Jiang, Tianze},
  title = {Global Universality of Singular Values in Products of Many Large Random Matrices},
  year = {2025},
  eprint = {2503.07872v1},
  archivePrefix = {arXiv},
  primaryClass = {math.PR},
  url = {https://arxiv.org/abs/2503.07872v1}
}

@article{tao2015local,
  author = {Tao, Terence and Vu, Van},
  title = {Random matrices: Universality of local spectral statistics of non-Hermitian matrices},
  journal = {The Annals of Probability},
  volume = {43},
  number = {2},
  pages = {782--874},
  year = {2015},
  doi = {10.1214/13-AOP876},
  eprint = {1206.1893},
  archivePrefix = {arXiv}
}

\end{document}